\documentclass[11pt,b5paper,notitlepage]{article}
\usepackage[b5paper, margin={0.5in,0.65in}]{geometry}

\usepackage{amsmath,amscd,amssymb,amsthm,mathrsfs,amsfonts,layout,indentfirst,graphicx,caption,mathabx, stmaryrd,appendix,calc,imakeidx,upgreek} 
\usepackage{palatino}  
\usepackage{slashed} 
\usepackage{mathrsfs} 
\usepackage{extarrows} 
\usepackage{enumitem} 

\usepackage{fancyhdr} 

\usepackage{ulem}  

\usepackage{tikz-cd}
\usepackage[nottoc]{tocbibind}   

\makeindex

\usepackage{lipsum}
\let\OLDthebibliography\thebibliography
\renewcommand\thebibliography[1]{
	\OLDthebibliography{#1}
	\setlength{\parskip}{0pt}
	\setlength{\itemsep}{2pt} 
}

\allowdisplaybreaks  
\usepackage{latexsym}
\usepackage{chngcntr}
\usepackage[colorlinks,linkcolor=blue,anchorcolor=blue, linktocpage,
]{hyperref}
\hypersetup{ urlcolor=cyan,
	citecolor=[rgb]{0,0.5,0}}

\counterwithin{figure}{section}

\theoremstyle{definition}
\newtheorem{df}{Definition}[section]
\newtheorem{eg}[df]{Example}

\newtheorem{rem}[df]{Remark}
\newtheorem{ass}[df]{Assumption}

\theoremstyle{plain}
\newtheorem{thm}[df]{Theorem}

\newtheorem{pp}[df]{Proposition}
\newtheorem{co}[df]{Corollary}
\newtheorem{lm}[df]{Lemma}

\newcommand{\fk}{\mathfrak}
\newcommand{\mc}{\mathcal}
\newcommand{\wtd}{\widetilde}

\newcommand{\wch}{\widecheck}
\newcommand{\ovl}{\overline}

\newcommand{\Tr}{\mathrm{Tr}}
\newcommand{\End}{\mathrm{End}} 
\newcommand{\id}{\mathbf{1}}
\newcommand{\Hom}{\mathrm{Hom}}
\newcommand{\Conf}{\mathrm{Conf}}
\newcommand{\Res}{\mathrm{Res}}

\newcommand{\Span}{\mathrm{Span}}

\newcommand{\bk}[1]{\langle {#1}\rangle}

\newcommand{\scr}{\mathscr}

\newcommand{\im}{\mathbf{i}}

\newcommand{\sgm}{\varsigma}
\newcommand{\SX}{{S_{\fk X}}}

\newcommand{\mbb}{\mathbb}
\newcommand{\mbf}{\mathbf}
\newcommand{\blt}{\bullet}

\newcommand{\Vbb}{\mathbb V}
\newcommand{\Ubb}{\mathbb U}

\newcommand{\Wbb}{\mathbb W}
\newcommand{\Mbb}{\mathbb M}
\newcommand{\Gbb}{\mathbb G}
\newcommand{\Cbb}{\mathbb C}
\newcommand{\Nbb}{\mathbb N}
\newcommand{\Zbb}{\mathbb Z}
\newcommand{\Pbb}{\mathbb P}

\newcommand{\Ebb}{\mathbb E}
\newcommand{\cbf}{\mathbf c}
\newcommand{\wt}{\mathrm{wt}}

\newcommand{\btl}{\blacktriangleleft}
\newcommand{\btr}{\blacktriangleright}

\newcommand{\SXb}{{S_{\fk X_b}}}
\newcommand{\pr}{\mathrm {pr}}
\newcommand{\SXtd}{S_{\wtd{\fk X}}}

\numberwithin{equation}{section}

\title{Sewing and Propagation of Conformal Blocks}
\author{{\sc Bin Gui}
}
\date{}
\begin{document}\sloppy 
	\pagenumbering{arabic}

	\maketitle

\newcommand\blfootnote[1]{%
	\begingroup
	\renewcommand\thefootnote{}\footnote{#1}%
	\addtocounter{footnote}{-1}%
	\endgroup
}


\tableofcontents





	
	

	

\newpage

\begin{abstract}
Propagation is a standard way of producing certain new conformal blocks from old ones that corresponds to the geometric procedure of adding new  distinct points to a pointed compact Riemann surface. On the other hand, sewing conformal blocks corresponds to sewing compact Riemann surfaces. 

In this article, we clarify the relationships between these two procedures. Most importantly, we show that "sewing and propagation are commuting procedures". More precisely: let $\upphi$ be a conformal block associated to a vertex operator algebra $\Vbb$ and a compact Riemann surface to be sewn, and let $\wr^n\upphi$ be its $n$-times propagation. If the  sewing $\wtd{\mc S}\upphi$  converges, then $\wtd{\mc S}\wr^n\upphi$ (the sewing of $\wr^n\upphi$) automatically converges, and it equals $\wr^n\wtd{\mc S}\upphi$ (the $n$-times propagation of the sewing $\wtd{\mc S}\upphi$). 

The proof of this result relies on establishing the propagation of conformal blocks associated to \textit{holomorphic} families of compact Riemann surfaces. We prove this in our paper using the idea that ``propagation is itself a sewing followed by an analytic continuation". This result generalizes previous ones on single Riemann surfaces \cite{Zhu94,FB04}, and supplements those on \textit{algebraic} families of complex algebraic curves \cite{Cod19,DGT19a}.

The results in this paper will be used in \cite{Gui21} as the main technical tools to relate the (genus-$0$) permutation-twisted $\Vbb^{\otimes k}$-conformal blocks (i.e. intertwining operators) and the untwisted $\Vbb$-conformal blocks (of possibly higher genera).
\end{abstract}

\section{Introduction}

\subsection*{Propagating conformal blocks}

Let $\Vbb$ be a vertex operator algebra (VOA) with vacuum vector $\id$. Let $\fk X=(C;x_1,\dots,x_N;\eta_1,\dots,\eta_N)$ be an $N$-pointed compact Riemann surface with local coordinates, namely, each connected component of the compact Riemann surface $C$ contains at least one of the distinct marked points $x_1,\dots,x_N$, and each $\eta_j$ is an injective holomorphic function on a neighborhood of $x_j$ sending $x_j$ to $0$ (i.e., an (analytic) local coordinate at $x_j$). Associate to each $x_j$ a $\Vbb$-module $\Wbb_j$. Then a conformal block  associated to $\fk X$ and $\Wbb_\blt=\Wbb_1\otimes\cdots\otimes\Wbb_N$ is a linear functional $\upphi:\Wbb_\blt\rightarrow\Cbb$ ``invariant" under the actions of $\Vbb$ (Cf. \cite{Zhu94,FB04,DGT19a}). When $C$ is the Riemann sphere $\Pbb^1$, the simplest examples of conformal blocks are as follows. (We let $\zeta$ be the standard coordinate of $\Cbb$.)

\begin{enumerate}
\item $\fk X=(\Pbb^1;0;\zeta)$, $\Wbb$ is associated to the marked point $0$. Then each $T\in\Hom_\Vbb(\Wbb,\Vbb')$ (where $\Vbb'$ is the contragredient module of the vacuum $\Vbb$) provides a conformal block
\begin{align*}
w\in\Wbb\mapsto \bk{Tw,\id}	
\end{align*}
Here $\bk{\cdot,\cdot}$ refers to the standard pairing of $\Vbb$ and $\Vbb'$. Of particular interest is the case that an isomorphism of $\Vbb$-modules $T:\Vbb\xrightarrow{\simeq}\Vbb'$ exists and is fixed. Then there is a canonical conformal block associated to $\fk X$ and $\Vbb$.
\item $\fk X=(\Pbb^1;0,\infty;\zeta,\zeta^{-1})$, and $\Wbb,\Wbb'$ are associated to $0,\infty$. Then  we have a conformal block
\begin{align}
\uptau_\Wbb:\Wbb\otimes\Wbb'\rightarrow\Cbb,w\otimes w'\mapsto \bk{w,w'}.	\label{eq58}
\end{align}
\item $\fk X=(\Pbb^1;0,z,\infty;\zeta,\zeta-z,\zeta^{-1})$, and $\Wbb,\Vbb,\Wbb'$ are associated to $0,z,\infty$. The the vertex operation $Y$ for $\Wbb$ defines a conformal block
\begin{align}
w\otimes v\otimes w'\in\Wbb\otimes\Vbb\otimes\Wbb'\mapsto \bk{Y(v,z)w,w'}.\label{eq55}
\end{align}
\end{enumerate}

Now, we add a new point $y\in C\backslash\{x_1,\dots,x_N\}$ (together with a local coordinate $\mu$) to $\fk X$ and call this new data $\wr\fk X_y$, and associate the vacuum module $\Vbb$ to $y$. Then  each conformal block $\upphi:\Wbb_\blt\rightarrow\Cbb$ associated to $\fk X$ and $\Wbb_\blt$ canonically gives rise to one $\wr\upphi_y:\Vbb\otimes\Wbb_\blt\rightarrow\Cbb$ associated to $\wr\fk X_y$ and $\Vbb\otimes\Wbb_\blt$, called the \textbf{propagation} of $\upphi$ at $y$. The propagation is uniquely determined by the fact that
\begin{align}
\wr\upphi(\id\otimes w_\blt)_y=\upphi(w_\blt).\label{eq53}
\end{align}
For example, it follows easily from such uniqueness that the third example above is the propagation of the second one at $z$, i.e.
\begin{align*}
\wr\uptau_\Wbb(v\otimes w\otimes w)_z=	\bk{Y(v,z)w,w'}.
\end{align*} 
More generally, when $y\in C$ is close to $x_i$ and the local coordinate $\mu$ at $y$ is $\eta_i-\eta_i(y)$, 
\begin{align}
\wr\upphi(v\otimes w_\blt)_y=\upphi(w_1\otimes\cdots\otimes Y(v,\eta_j(y))w_i\otimes\cdots\otimes w_N)	\label{eq54}
\end{align}
where the right hand side converges absolutely as a formal Laurent series of $\eta_j(y)$. (Cf. \cite[Thm. 6.2]{Zhu94}, \cite[Chapter 10]{FB04}, or Thm. \ref{lb4} of this article.) The uniqueness of $\wr\upphi$ satisfying \eqref{eq53} is not hard to show; what is more difficult is to prove the existence of propagation (cf. \cite{TUY89,Zhu94,Zhu96,FB04,Cod19,DGT19a}).

\subsection*{Sewing conformal blocks}

It is worth noting that the right hand side of \eqref{eq54} is the sewing of $\upphi$ and $\wr\tau_\Wbb$ ($=$the conformal block defined in \eqref{eq55}) corresponding the geometric sewing of $C$ and $\Pbb^1$ along the points $x_i,\infty$ with respect to their local coordinates $\eta_i,\zeta^{-1}$. In general, given an $(N+2)$-pointed compact Riemann surface with local coordinates $\wtd{\fk X}=(\wtd C;x_1,\dots,x_N,x',x'';\eta_1,\dots,\eta_N,\xi,\varpi)$ where each connected component of $\wtd C$ intersects $\{x_1,\dots,x_N\}$, if $\xi$ (resp. $\varpi$) is defined on a neighborhood $W'$ of $x'$ (resp. $W''$ of $x''$) such that $\xi(W')$ is the open disc $\mc D_r$ with radius $r$ (resp. $\varpi(W'')=\mc D_\rho$), and that $W'$ (resp. $W''$) contains only one point among $x_1,\dots,x_N,x',x''$. Then for each $0<|q|<r\rho$, we remove
\begin{align*}
F'=\{y\in W':|\xi(y)|\leq |q|/\rho\},\qquad F''=\{y\in W'':|\varpi(y)|\leq |q|/r\},	
\end{align*}
from $\wtd C$, and glue the remaining part by identifying all $y'\in W'$ with $y''\in W''$ if $\xi(y')\varpi(y'')=q$. As a result, we obtain a new compact Riemann surface $\mc C_q$ with marked points $x_1,\dots,x_N$ and local coordinates $\eta_1,\dots,\eta_N$. We denote this data by $\fk X_q$. Corresponding to this geometric sewing, we associated $\Vbb$-modules $\Wbb_1,\dots,\Wbb_N,\Mbb,\Mbb'$ to the marked points $x_1,\dots,x_N,x',x''$ where $\Mbb'$ is contragredient to $\Mbb$, and assume that the modules are $\Nbb$-gradable (i.e., admissible) with grading operator $\wtd L_0$ such that each graded subspace is finite-dimensional. $q^{\wtd L_0}\in\End(\Mbb)[[q]]$ can be regarded as an element of  $\Mbb\otimes\Mbb'[[q]]$, which we denote by $q^{\wtd L_0}\btr\otimes\btl$. If $\uppsi:\Wbb_\blt\otimes\Mbb\otimes\Mbb'\rightarrow\Cbb$ is a conformal block associated to $\wtd{\fk X}$, we define a linear $\wtd{\mc S}\uppsi:\Wbb_\blt\rightarrow\Cbb[[q]]$ sending each $w\in\Wbb_\blt$ to
\begin{gather}
	\wtd{\mc S}\uppsi(w)=\uppsi(w\otimes q^{\wtd L_0}\btr\otimes\btl).\label{eq56}
\end{gather}
It was shown in \cite[Thm. 8.5.1]{DGT19b} that the above linear map defines a ``formal conformal block" (i.e., a ``conformal block" when $q$ is infinitesimal). If this series converges absolutely on $|q|<r\rho$, then it defines an actual conformal block associated to $\fk X_q$ \cite[Thm. 11.2]{Gui23}, called the \textbf{sewing} of $\uppsi$.

In the above process, if $\wtd C$ is connected, then $\mc C_q$ is the self-sewing of $\wtd C$. For instance, if we sew the  $\fk X$ in the above example 3 along $0$ and $\infty$ to get a torus, we accordingly sew the conformal block \eqref{eq55} to obtain the (normalized) character of $\Wbb$-module $v\mapsto \Tr (Y(v,z)q^{\wtd L_0})$, which plays an important role in the early development of VOA theory. If $\wtd C$ has two connected component $\wtd C_1,\wtd C_2$, and if we sew $\wtd C$ along $x'\in\wtd C_1,x''\in\wtd C_2$, we obtain a connected sum of $\wtd C_1$ and $\wtd C_2$. If we choose $\wtd C=C\sqcup\Pbb^1$ and sew $\wtd C$ along $x_i\in C$ and $\infty\in\Pbb^1$, then at $q=1$, the corresponding sewing of the conformal blocks $\upphi$ and \eqref{eq55} is just \eqref{eq54}, and the new Riemann surface we get is naturally equivalent to $C$.

\subsection*{Propagation is a sewing followed by an analytic continuation}

Now,  \eqref{eq54} indicates that propagation and sewing are related. Roughly speaking, propagation can be understood as follows: \uline{When the inserted point $y$ is close to a marked point $x_i$, the propagation is defined by sewing. When $y$ is far from the marked points, the propagation is defined by analytic continuation (provided that it exists).} (See Exp. \ref{lb9} and the proof of Thm. \ref{lb4} for details.)

The above important point is implicit in the literature (such as \cite{Zhu94,Zhu96}). However, it seems that no existing result relies completely on this idea (and especially on sewing) to establish the existence of propagation.  The first result for general VOAs over Riemann surfaces is due to Zhu \cite{Zhu94}. Zhu used analytic methods to establish the propagation of conformal blocks over a single compact Riemann surface. However, instead of using sewing in the proof, he constructed propagation using certain ``Verma modules". (See \cite[Thm. 6.1]{Zhu94})   A more algebraic approach was later given in \cite[Thm. 10.3.1]{FB04}. Propagation over \textit{algebraic} families of complex algebraic curves was given in \cite[Thm. 3.6]{Cod19} and \cite[Thm. 5.1]{DGT19a}. Instead of using sewing in their proofs, they used a PBW basis instead. (In fact, the analytic sewing is unavailable in algebraic geometry.)

Unlike previous approaches to propagation, ours is based largely on the above understanding of propagation (i.e. it is sewing+analytic continuation). Let us explain it in more details below.

\subsection*{Main result: propagation of analytic families of conformal blocks}

The first main result of this article is Thm. \ref{lb4}, which establishes the propagation of conformal blocks for a \textit{holomorphic} family of compact Riemann surfaces with marked points. Roughly speaking, Thm. \ref{lb4} says the following: Suppose that we have a (holomorphic) family $\fk X$ of compact Riemann surfaces with marked points. Let $\mc B$ be a base manifold with holomorphic parameters $\tau_\blt=(\tau_1,\dots,\tau_m)$. Suppose that a holomorphic section $\varphi=\varphi(\tau_\blt)$ of conformal blocks associated to $\fk X$ is given. Then its propagation $\wr\varphi=\wr\varphi(\tau_\blt,z)$ exists as a section which is simultaneously holomorphic with respect to $\tau_\blt$ and $z$. Here $z$ is (locally) a parameter on the fibers of Riemann surfaces.  

I have mentioned that propagation over a single Riemann surface was already proved in the literature. However, one cannot use fiberwise propagation to construct $\wr\varphi$ for a family of surfaces: it implies only that $\wr\varphi(\tau_\blt,z)$ is holomorphic over $z$ for each fixed $\tau_\blt$, but not that $\wr\varphi(\tau_\blt,z)$ is holomorphic over $\tau_\blt$ for each fixed $z$ (even if we know that $\varphi(\tau_\blt)$ is holomorphic over $\tau_\blt$).

\subsection*{Ideas of the proof of Thm. \ref{lb4}}

Compare to the results already existing in the literature, the most novel part in Thm. \ref{lb4} is that \uline{the holomorphicity of $\varphi(\tau_\blt)$ with respect to $\tau_\blt$ implies the (simultaneous) holomorphicity of $\wr\varphi(\tau_\blt,z)$ with respect to $\tau_\blt,z$}. Our main tool for proving this fact is the \textbf{strong residue theorem} for holomorphic families of compact Riemann surfaces, given in Thm. \ref{lb7}. Roughly speaking, the (classical) strong residue theorem says the following: Suppose that to each marked point of a compact Riemann surface a formal Laurent series is associated. Then these formal Laurent series are expansions of a meromorphic section with possible poles only at these marked points if and only if it satisfies the residue theorem when multiplied by any meromorphic $1$-forms with possible poles only at these points. The classical strong residue theorem is an easy application Serre's duality for holomorphic bundles on compact Riemann surfaces (cf. \cite[Sec. 1.2.3]{Ueno08}). Its generalization to holomorphic families (i.e. Thm. \ref{lb7}) is more involved: our proof combines Serre's duality with Grauert's base change theorem in an appropriate way.

Another feature of our proof of Thm. \ref{lb4} is that \textit{we uses essentially the viewpoint that propagation is a sewing followed by an analytic continuation}. In fact, in the proof of Thm. \ref{lb4}, we first establish the existence of $\wr\varphi(\tau_\blt,z)$ when $z$ is neared a marked point. We prove this part by using the fact that $\wr\varphi(\tau_\blt,z)$ is a sewing of $\varphi(\tau_\blt)$ and a $3$-pointed genus $0$ conformal block. Then we perform the analytic continuation. The details of this sewing construction are given in Exp. \ref{lb9}. 

In particular, this viewpoint allows us to prove that $\wr\varphi$ is a conformal block (but not just an arbitrary element) when $z$ is near the marked points by using the nontrivial fact that the sewing of a conformal block is again a conformal block as long as the sewing is convergent, cf. Thm. \ref{lb8}. (This theorem was originally proved in \cite[Thm. 11.3]{Gui23}.) It is precisely this part that plays the role of the construction of certain Verma modules in the proof of \cite{Zhu94}, and that of the PBW basis in the proofs of \cite{Cod19,DGT19a}. Similarly, in the proof of Thm. \ref{lb4} we also need the fact that the analytic continuation of a conformal block is a conformal block, as proved in \cite[Prop. 6.4]{Gui23} (cf. Prop. \ref{lb10}).

\subsection*{Main result: sewing and propagation are commuting}

Using Thm. \ref{lb4} and induction, it is now easy to prove Thm. \ref{lb21}, another main result of this article, which says roughly that ``sewing commutes with propagation".


Roughly speaking, Thm. \ref{lb21} says the following: Suppose that $\uppsi$ is a conformal block associated to a compact Riemann surface with marked point. Suppose that the sewing $\wtd{\mc S}\uppsi=\eqref{eq56}$ along a pair of marked points is convergent. Then we have
\begin{align}
\wr^n\wtd{\mc S}\uppsi=\wtd{\mc S}\wr^n\uppsi,	\label{eq57}
\end{align}
where both sides are well-defined holomorphic sections.

Note that Thm. \ref{lb21} shows, in particular, that the convergence of $\wtd{\mc S}\uppsi$ implies automatically the convergence/analyticity of the sewing $\wtd{\mc S}\wr^n\uppsi$. This nontrivial phenomenon first appeared in a prominent way in \cite{Zhu96}. In that paper, Zhu first used differential equations to establish the analyticity of $1$-pointed conformal blocks of genus-$1$ for $C_2$-cofinite VOAs. Then he used his recurrent formula to prove that the $n$-pointed conformal blocks of genus-$1$ are analytic when all the marked points are associated with the vacuum module $\Vbb$. His proof of ``$1$-pointed convergence/analyticity implies $n$-pointed analyticity" in genus $1$ by recurrent formula does not rely on $C_2$-cofiniteness or differential equations. However, this phenomenon in higher genus was not further investigated in \cite{Zhu94}.

In the proof of Thm. \ref{lb21}, the analyticity of $\wtd{\mc S}\wr^n\uppsi$ follows from that of $\wr^n\wtd{\mc S}\uppsi$ (since they are locally equal as formal power series). The analyticity of $\wr^n\wtd{\mc S}\uppsi$ is an easy consequence of Thm. \ref{lb4} (applied inductively to the propagation of the holomorphic family of conformal blocks $\wr^{n-1}\wtd{\mc S}\uppsi$).


\subsection*{Applications}

We give an application of Thm. \ref{lb21}. Let $\fk Y=(C;x_1,\dots,x_N;\eta_1,\dots,\eta_N)$, associate $\Wbb_j$ to $x_j$ for each $j$, and choose a conformal block $\upphi:\Wbb_\blt\rightarrow\Cbb$ associated to $\fk Y$. Choose $1\leq i\leq N$. Let $\fk P=(\Pbb^1;0,\infty;\zeta,\zeta^{-1})$, and associate $\Wbb_i,\Wbb_i'$ to $0,\infty$. Then $\uppsi:=\upphi\otimes\uptau_{\Wbb_i}:\Wbb_\blt\otimes\Wbb_i\otimes\Wbb_i'\rightarrow\Cbb$ (recall \eqref{eq58}) is a conformal block associate to the disjoint union $\wtd{\fk X}=\fk Y\sqcup\fk P$. If we sew $\wtd{\fk X}$ along $x_i\in C$ and $\infty\in\Pbb^1$ at $q$, the new pointed Riemann surface with local coordinates $\fk X_q$ is
\begin{align*}
\fk X_q=(C;x_1,\dots,x_N;\eta_1,\dots,q^{-1}\eta_i,\dots,\eta_N),	
\end{align*}
and  (setting $w_\blt=w_1\otimes\cdots\otimes w_N$ as usual)
\begin{align}
\wtd{\mc S}\uppsi(w_\blt)=\upphi(w_1\otimes\cdots\otimes q^{\wtd L_0}w_i\otimes\cdots\otimes w_N),\label{eq59}
\end{align}
which clearly converges absolutely for all $q$. Assume $\eta_i$ is defined on an open disc $W_i\ni x_i$ such that $\eta_i(W_i)=\mc D_{r_i}$ has radius $r_i$, and that $W_i$ contains only $x_i$ among $x_1,\dots,x_N$. Choose $r>0$. Then, according to our main result, the sewing of $n$-propagation
\begin{align*}
&\wtd{\mc S}\wr^n\uppsi	(v_1\otimes\cdots\otimes v_n\otimes w_\blt)_{\eta_i^{-1}(qz_1),\dots,\eta_i^{-1}(qz_n)}\\
=&\upphi(w_1\otimes\cdots\otimes w_{i-1} \otimes q^{\wtd L_0}\btr\otimes w_{i+1}\otimes\cdots\otimes w_N)\cdot\wr^n\uptau_{\Wbb_i}(v_1\otimes\cdots\otimes v_n\otimes w_i\otimes\btl)_{z_1,\dots,z_n}
\end{align*}
(assuming that the local coordinate at each $z_j\in\Pbb^1$ is $\zeta-z_j$, and the one at $\eta_i^{-1}(qz_j)\in C$ is $q^{-1}\eta_i-z_j$) converges absolutely and uniformly when $z_1,\dots,z_n$ vary on any compact set of the configuration space $\Conf^n(\mc D_r^\times)$ (where $\mc D_r^\times=\{z\in\Cbb:0<|z|<r\}$) and when $|q|<r_i/r$.

We are especially interested in the case that $q=1$, which is accessible when $r<r_i$, namely, when $0<|z_1|,\dots,|z_n|<r_i$. Then $\wr^n\wtd{\mc S}\uppsi=\wtd{\mc S}\wr^n\uppsi$ implies (notice \eqref{eq59})
\begin{align}
&\wr^n\upphi(v_1\otimes\cdots\otimes v_n\otimes w_\blt)_{\eta_i^{-1}(z_1),\dots,\eta_i^{-1}(z_n)}\nonumber\\
=&\upphi(w_1\otimes\cdots\otimes w_{i-1} \otimes \btr\otimes w_{i+1}\otimes\cdots\otimes w_N)\cdot\wr^n\uptau_{\Wbb_i}(v_1\otimes\cdots\otimes v_n\otimes w_i\otimes\btl)_{z_1,\dots,z_n}.\label{eq60}
\end{align}
In the special case that $0<|z_1|<\cdots<|z_n|<r_i$, the above relation becomes
\begin{align}
&\wr^n\upphi(v_1\otimes\cdots\otimes v_n\otimes w_\blt)_{\eta_i^{-1}(z_1),\dots,\eta_i^{-1}(z_n)}\nonumber\\
=&\upphi\big(w_1\otimes\cdots\otimes Y(v_n,z_n)\cdots Y(v_1,z_1)w_i\otimes\cdots\otimes w_N\big)\label{eq61}	
\end{align}
where the right hand side converges absolutely. Zhu proved relation \eqref{eq61} in \cite[Thm. 6.2]{Zhu94} when $v_1,\dots,v_n$ are primary, or when the local coordinates are contained in a projective structure (i.e., an atlas whose transition functions are M\"obius transforms). But, as explained below, the general case, especially when $0<|z_1|=\cdots=|z_n|<r_i$, is also important.

Take an automorphism $g$ of $\Vbb^{\otimes k}$ to be the permutation associated to the cycle $(12\cdots k)$. Starting from a $\Vbb$-module $\Wbb$, Barron-Dong-Mason constructed in \cite{BDM02} a (canonical) $g$-twisted $\Vbb^{\otimes k}$-module structure on the same vector space $\Wbb$. In particular, they explicitly described the twisted vertex operator $Y^g$ for vectors in $\Vbb^{\otimes k}$ of the form $\id\otimes\cdots\otimes v\otimes\cdots\otimes\id$. For an arbitrary vector of $\Vbb^{\otimes k}$, the $Y^g$  can then be described using normal ordering. Their proof that $Y^g$ satisfies the axioms of a $g$-twisted module is algebraic, and in particular relies on a previous algebraic result of \cite{Li96}. Recently, another algebraic proof was given by Dong-Xu-Yu in \cite{DXY21} using Zhu's algebras.

Now, our observation in this article is that since $Y^g(v_1\otimes\cdots\otimes v_k,z)$ can be described by $\wr^k\uptau_{\Wbb_i}$ at $(z_1,\dots,z_k)$ (where $z_1,\dots,z_k$ are the distinct $k$-th roots of unity of $z$), using the consequences of our main result such as relation \eqref{eq60}, we can give a geometric and complex-analytic proof that $Y^g$ satisfies the axioms of a $g$-twisted module. Namely, we check that $Y^g$ satisfies the complex-analytic version of Jacobi identity (as in \cite{Hua10}). See Sec. \ref{lb24} for details. Our proof is in the same spirit as checking the Jacobi identity for VOA modules using contour integrals. Note that the geometric meaning of Barron-Dong-Mason's construction of twisted modules was given in \cite{BDM02,BHL07}, but the verification in \cite{BDM02} that these twisted modules satisfy Jacobi identity is purely algebraic. The merit of our approach, on the other hand, is that \textit{we use geometric methods to prove results about mathematical objects with geometric origins}.

Moreover,   our complex-analytic method will be generalized in \cite{Gui21} to construct permutation twisted conformal blocks from untwisted ones, and vice versa. As an important consequence, the fusion rules for permutation twisted modules of a strongly rational VOA will be completely determined in \cite{Gui21}.

\subsection*{Outline}

This article is organized as follows. In Section 2, we fix the geometric notations used in later sections, and define the (multi) propagations for an (analytic) family of compact Riemann surfaces. In the case of a single compact Riemann surface $C$ with marked points $S=\{x_1,\dots,x_N\}$, its $n$-propagation is easy to describe: If we let several distinct points $y_1,\dots,y_n$ move on $C\setminus S$, we obtain a family of compact Riemann surfaces (all isomorphic to $C$) with $N$ fixed marked points and $n$ varying points over the base manifold $\Conf^n(C\setminus S)$.

We recall the definitions and basic properties of sheaves of VOAs (i.e. VOA bundles) and conformal blocks in Sections 3 and 4. In Section 5, we recall some important facts about the sewing of conformal blocks associated to the sewing of a family of compact Riemann surfaces. In Section 6, we relate sheaves of VOAs and the $\scr W$-sheaves which were naturally introduced to define (sheaves of) conformal blocks.

In Section 7, we give a new proof of conformal block propagation for (analytic) families of compact Riemann surfaces. In particular, we prove that propagation is compatible (in the complex analytic sense) with the deformation of pointed compact Riemann surfaces. Roughly speaking, this means that if the original conformal blocks are parametrized by $\tau\in\mc B$ where $\mc B$ is the base manifold of the family, and if the propagation on each fiber is parametrized by $z$, then the propagation is a multivariable analytic function of $(z,\tau)$. The precise statement is formulated using the language of sheaves; see Thm. \ref{lb4}. These results were proved in \cite[Thm. 3.6]{Cod19} for CFT type VOAs using a Lie-theoretic method, which relies on the fact that such VOAs have PBW bases. As explained earlier, our proof is based on the idea of sewing, and relies on the Strong Residue Theorem and the fact that the sewing of conformal blocks are conformal blocks \cite[Thm. 11.2]{Gui23}, whose  formal version was proved in \cite{DGT19b}. 

Note that here we should use the Strong Residue Theorem for analytic families of compact Riemann surfaces. This result is well-known, although we are not able to find a proof in the literature. So we include a proof in the Appendix Section A.

We discuss elementary properties of multi-propagation in Section 8.  Most of these important properties were more or less known before (cf. \cite{FB04}) but not explicitly written down. We collect these results under Thm. \ref{lb17} so that they can be directly cited or used in future works on VOA. These results follow rather directly from those in the previous sections.

The main theorem of this article, summarized by the slogan ``sewing commutes with propagation", is proved in Section 9. To give an application of this result, we construct in Section 10 permutation-twisted modules for tensor product VOAs.

\subsection*{Acknowledgment}

I would like to thank Nicola Tarasca for helpful discussions. The author was supported by BMSTC and ACZSP (Grant no. Z221100002722017).

\section{The geometric setting}\label{lb2}

We set $\Nbb=\{0,1,2,\dots\}$ and $\Zbb_+=\{1,2,3,\dots\}$. Let $\Cbb^\times=\Cbb\setminus\{0\}$. \index{C@$\Cbb^\times$} For each $r>0$, we let $\mc D_r=\{z\in\Cbb:|z|<r\}$ and $\mc D_r^\times=\mc D_r\setminus\{0\}$. \index{Dr@$\mc D_r,\mc D_r^\times$} For any topological space $X$, we define the configuration space $\Conf^n(X)=\{(x_1,\dots,x_N)\in X^n:x_i\neq x_j~\forall 1\leq i<j\leq n\}$. \index{Conf@$\Conf^n(X)$}

For each complex manifold $X$, $\scr O_X$ is the sheaf of holomorphic functions of $X$. For each $x\in X$ and any $\scr O_X$-module $\scr E$, $\scr E_x$ is the stalk of $\scr E$ at $x$. $\fk m_{X,x}$ (or simply $\fk m_x$ when no confusion arises)  is by definition $\{f\in\scr O_{X,x}:f(x)=0\}$. $\scr E|_x:=\scr E_x/\fk m_x\scr E_x\simeq\scr E\otimes_{\scr O_X}\scr O_{X,x}/\fk m_x$ is the fiber of $\scr E$ at $x$. More generally, if $Y$ is a closed complex sub-manifold of $X$ with $\scr I_Y$ being the ideal sheaf (the sheaf of all sections of $\scr O_X$ vanishing at $Y$), then the restriction $\scr E|_Y$ is defined to be $\scr E\otimes_{\scr O_X}\scr O_X/\scr I_Y$ (restricted to the set $Y$). We suppress the subscript $\scr O_X$ under $\otimes$ when taking tensor products of $\scr O_X$-modules. If $s$ is a section of $\scr E$, then $s|_Y$ is the corresponding value $s\otimes 1$ in $\scr E|_Y$.

(For the readers not familiar with the language of sheaf of modules: we only consider the case that $\scr E$ is locally free (with finite or infinite rank), i.e., a holomorphic vector bundle. Then $\scr E|_Y$ resp. $s|_Y$ is the usual restriction of the vector bundle resp. vector field to the submanifold $Y$.)

If $\scr E$ is locally free, $\scr E^\vee$ denotes its dual vector bundle.

For a Riemann surface $C$, its cotangent line bundle is denoted by $\omega_C$. \index{zz@$\omega_C$, $\omega_{\mc C_b}$}

A \textbf{family of compact Riemann surfaces} $\fk X$ is by definition a holomorphic proper  map of complex manifolds
\begin{align*}
\fk X=(\pi:\mc C\rightarrow\mc B)	
\end{align*}
that is a submersion and satisfies that each fiber $\mc C_b:=\pi^{-1}(b)$ (where $b\in\mc B$) is a (non-necessarily connected) compact Riemann surface.  



A \textbf{family of $N$-pointed compact Riemann surfaces} is by definition
\begin{align}
	\fk X=(\pi:\mc C\rightarrow\mc B;\sgm_1,\dots,\sgm_N)\label{eq2}	
\end{align}
where $\pi:\mc C\rightarrow\mc B$ is a family of compact Riemann surfaces, each section $\sgm_j:\mc B\rightarrow\mc C$ is holomorphic and satisfies $\pi\circ\sgm_j=\id_{\mc B}$, and any two $\sgm_i(\mc B),\sgm_j(\mc B)$ (where $1\leq i<j\leq N$) are disjoint. Unless otherwise stated, we also assume that every connected component of each fiber \index{CbXb@$\mc C_b,\fk X_b$}
\begin{align*}
	\mc C_b=\pi^{-1}(b)	
\end{align*}
(where $b\in\mc B$) contains at least one of $\sgm_1(b),\dots,\sgm_N(b)$. We set
\begin{align*}
	\fk X_b=(\mc C_b;\sgm_1(b),\dots,\sgm_N(b)),	
\end{align*}
which is an $N$-pointed compact Riemann surface. We define closed submanifold \index{SX@$\SX$}
\begin{equation*}
	\SX=\bigcup_{j=1}^N\sgm_j(\mc B),	
\end{equation*}
considered also as a divisor of $\mc C$. For any sheaf of $\scr O_{\mc C}$-module $\scr E$, and for any $n\in\Zbb$, we set \index{ES@$\scr E(n\SX),\scr E(\star\SX)$}
\begin{gather*}
	\scr E(n\SX):=\scr E\otimes\scr O_{\mc C}(n\SX),\\
	\scr E(\star\SX)=\varinjlim_{n\in\Nbb}\scr E(n\SX).
\end{gather*}
When $\scr E$ is a vector bundle, $\scr E(n\SX)$ is the sheaf of sections of $\scr E$ which possibly has poles at each $\sgm_j(\mc B)$ with order at most $n$.

For each $1\leq j\leq N$, a \textbf{local coordinate} of $\fk X$ at $\sgm_j$ is defined to be a holomorphic function $\eta_j\in\scr O(W_i)$ (where $W_i$ is a neighborhood of $\sgm_i(\mc B)$) which is injective on each fiber $W_i\cap\pi^{-1}(b)$ and has value $0$ on $\sgm_i(\mc B)$. It follows that $(\pi,\eta_j)$ is a biholomorphism from $W_i$ to a neighborhood of $\mc B\times\{0\}$ in $\mc B\times\Cbb$. $\eta_j|_{\mc C_b}$ is a local coordinate of the fiber $\mc C_b$ at the point $\sgm_j(b)$, which identifies a neighborhood  of $\sgm_j(b)$ (say $W_j\cap\mc C_b$) with an open subset of $\Cbb$ such that $\sgm_j(b)$ is identified with the origin. If $\fk X$ is equipped with local coordinates $\eta_1,\dots,\eta_N$ at $\sgm_1(\mc B),\dots,\sgm_N(\mc B)$ respectively, we set
\begin{align*}
\fk X_b=(\mc C_b;\sgm_1(b),\dots,\sgm_N(b);\eta_1|_{\mc C_b},\dots,\eta_N|_{\mc C_b}).	
\end{align*}
In particular, $S_{\fk X_b}=\sum_j\sgm_j(b)$ is a divisor of $\mc C_b$.

Now, we let $\fk X=\eqref{eq2}$ be $N$-pointed but not necessarily equipped with local coordinates. Define the \textbf{propagated family} $\wr\fk X$ as follows. Consider the commutative diagram 
\begin{equation*}
\begin{tikzcd}
\mc C\times_{\mc B}(\mc C\setminus\SX) \arrow[r] \arrow[d, "\wr\pi"] & \mc C \arrow[d, "\pi"] \\
\mc C\setminus\SX\arrow[r, "\pi"] & \mc B
\end{tikzcd}	
\end{equation*}
where $\mc C\times_{\mc B}(\mc C\setminus\SX)$ is the closed submanifold of $\mc C\times(\mc C\setminus\SX)$ consisting of all $(x,y)$ satisfying $\pi(x)=\pi(y)$, the first horizontal arrow is the projection onto the first component, and $\wr\pi$ is the projection onto the second component. \index{X@$\wr\fk X,\wr\mc C,\wr\mc B,\wr\pi$} We set
\begin{align*}
	\wr\mc B=\mc C\setminus\SX,\qquad \wr\mc C=\mc C\times_{\mc B}(\mc C\setminus\SX).	
\end{align*}
The holomorphic section $\sigma:\mc C\setminus\SX\rightarrow\mc C\times_{\mc B}(\mc C\setminus\SX)$ is set to be the diagonal map, i.e.,
\begin{align*}
\sigma:x\mapsto (x,x).	
\end{align*}
Define sections \index{zz@$\wr\sgm_j,\wr^n\sgm_j$}
\begin{align*}
\wr\sgm_j:\mc C\setminus\SX\rightarrow\mc C\times_{\mc B}(\mc C\setminus\SX),\qquad x\mapsto (\sgm_j\circ\pi(x),x).
\end{align*}
Then we obtain an $(N+1)$-pointed family $\wr\fk X$ of compact Riemann surfaces to be
\begin{align}
\wr\fk X=(\wr\pi:\wr\mc C\rightarrow\wr\mc B;\sigma,\wr\sgm_1,\dots,\wr\sgm_N).	\label{eq7}
\end{align}
Intuitively, $\wr\fk X$ is the result of adding one extra marked point to each fiber $\mc C_b$ disjoint from $\SXb$, letting this marked point vary on $\mc C_b\setminus\SXb$ over all $b\in\mc B$, and fixing the other marked points.

One can define multi-propagation inductively by $\wr^n\fk X=\wr\wr^{n-1}\fk X$, which corresponds to varying $n$ extra distinct points of $\mc C_b\backslash S_{\fk X_b}$. Write
\begin{align*}
\wr^n\fk X=(\wr^n\pi:\wr^n\mc C\rightarrow\wr^n\mc B;\sigma_1,\dots,\sigma_n,\wr^n\sgm_1,\dots,\wr^n\sgm_N).	
\end{align*}
Then $\wr^n\fk X$ can be described in a more explicit way. Let
\begin{align*}
\prod^n_{\mc B} \mc C\setminus\SX=\underbrace{(\mc C\setminus\SX)\times_{\mc B}\cdots\times_{\mc B}(\mc C\setminus\SX)}_{n}	
\end{align*}
which is the set of all $(x_1,\dots,x_n)\in\prod^n \mc C\setminus\SX$  satisfying $\pi(x_1)=\dots=\pi(x_n)$. Define the relative configuration space \index{ConfB@$\Conf^n_{\mc B}$}
\begin{align*}
\Conf^n_{\mc B}(\mc C\setminus\SX)=\Big\{(x_1,\dots,x_N)\in\prod\nolimits_{\mc B}^n \mc C\setminus\SX: x_i\neq x_j\text{ for any }1\leq i<j\leq n\Big\}
\end{align*}
which clearly admits a submersion $\Conf^n_{\mc B}(\mc C\setminus\SX)\rightarrow\mc B$ (sending each $(x_1,\dots,x_n)$ to $\pi(x_1)$). Take
\begin{align*}
\wr^n\pi:\mc C\times_{\mc B}	\Conf^n_{\mc B}(\mc C\setminus\SX)\rightarrow \Conf^n_{\mc B}(\mc C\setminus\SX).
\end{align*}
to be the pullback of $\pi:\mc C\rightarrow\mc B$ along $\Conf^n_{\mc B}(\mc C\setminus\SX)\rightarrow\mc B$. So we have a commutative diagram
\begin{equation*}
\begin{tikzcd}
\mc C\times_{\mc B}\Conf^n_{\mc B}(\mc C\setminus\SX) \arrow[r] \arrow[d, "\wr^n\pi"] & \mc C \arrow[d, "\pi"] \\
\Conf^n_{\mc B}(\mc C\setminus\SX)\arrow[r] & \mc B
\end{tikzcd}	
\end{equation*}
Then $\wr^n\fk X$ is equivalent to \index{Xn@$\wr^n\fk X,\wr^n\mc C,\wr^n\mc B,\wr^n\pi$}
\begin{align*}
\wr^n\fk X\simeq\Big(\wr^n\pi:\mc C\times_{\mc B}\Conf^n_{\mc B}(\mc C\setminus\SX)\rightarrow\Conf^n_{\mc B}(\mc C\setminus\SX);\sigma_1,\dots,\sigma_n,\wr^n\sgm_1,\dots,\wr^n\sgm_N\Big)	
\end{align*}
where \index{zz@$\wr\sgm_j,\wr^n\sgm_j$}
\begin{gather*}
\sigma_i(x_1,\dots,x_n)=(x_i,x_1,\dots,x_n),\\
\wr^n\sgm_j(x_1,\dots,x_n)=(\sgm_j\circ\pi(x_1),x_1,\dots,x_n)	
\end{gather*}
for each $1\leq i\leq n$, $1\leq j\leq N$, $(x_1,\dots,x_n)\in\Conf^n_{\mc B}(\mc C\setminus\SX)$.

\section{Sheaves of VOA}\label{lb1}

For any ($\Cbb$-)vector space $W$, we define four spaces of formal series \index{z@$[[z]],[[z^{\pm 1}]],((z)),\{z\}$}
\begin{gather*}
W[[z]]=\bigg\{\sum_{n\in\mathbb N}w_nz^n:\text{each }w_n\in W\bigg\},\\
W[[z^{\pm 1}]]=\bigg\{\sum_{n\in\mathbb Z}w_nz^n:\text{each }w_n\in W\bigg\},\\
W((z))=\Big\{f(z):z^kf(z)\in W[[z]]\text{ for some }k\in\mbb Z \Big\},\\
W\{z\}=\bigg\{\sum_{n\in\mathbb \Cbb}w_nz^n:\text{each }w_n\in W\bigg\}.
\end{gather*}

Throughout this article, $\Vbb$ is an $\Nbb$-graded vertex operator algebra (VOA) with vacuum $\id$ and conformal vector $\cbf$. We write $Y(v,z)=\sum_{n\in\Zbb}Y(v)_nz^{-n-1}$. Then $\{L_n=Y(\cbf)_{n+1}\}$ are Virasoro algebras, and $L_0$ gives grading $\Vbb=\bigoplus_{n\in\Nbb}\Vbb(n)$, where each $\Vbb(n)$ is finite-dimensional. 

In this article,  a $\Vbb$-module $\Wbb$ means a \textbf{finitely-admissible $\Vbb$-module}. This means that $\Wbb$ is a weak $\Vbb$-module in the sense of \cite{DLM97} with vertex operators $Y_\Wbb(v,z)=\sum_{n\in\Zbb}Y_\Wbb(v)_nz^{-n-1}$, that $\Wbb$ is equipped with a diagonalizable operator $\wtd L_0$ (not to be confused with $L_0=Y_\Wbb(\cbf)_1$ which is not necessarily diagonalizable!) satisfying \index{L0@$\wtd L_0$}  
\begin{align}\label{eq34}
[\wtd L_0,Y_\Wbb(v)_n]=Y_\Wbb(L_0 v)_n-(n+1)Y_\Wbb(v)_n,	
\end{align}
that the eigenvalues of $\wtd L_0$ are in $\Nbb$, and that each eigenspace $\Wbb(n)$ is finite-dimensional. Let \index{W@$\Wbb(n),\Wbb_{(n)}$}
\begin{align*}
	\Wbb=\bigoplus_{n\in\Nbb}\Wbb(n)	
\end{align*}
be the grading given by $\wtd L_0$. Each
\begin{align*}
	\Wbb^{\leq n}=\bigoplus_{0\leq k\leq n}	\Wbb{(k)}
\end{align*}
is finite-dimensional. We choose  the $\wtd L_0$ operator on $\Vbb$ to be $L_0$.

We can define the \textbf{contragredient $\Vbb$-module} $\Wbb'$ of $\Wbb$ as in \cite{FHL93}. We choose $\wtd L_0$-grading to be
\begin{align*}
\Wbb'=	\bigoplus_{n\in\Nbb}\Wbb'{(n)},\qquad \Wbb'{(n)}=\Wbb{(n)}^*.
\end{align*}
Therefore, if we let $\bk{\cdot,\cdot}$ be the pairing between $\Wbb$ and $\Wbb'$, then $\bk{\wtd L_0 w,w'}=\bk{w,\wtd L_0 w'}$ for each $w\in\Wbb,w'\in\Wbb'$.

The vertex operator  $Y_\Wbb$ for $\Wbb$ (abbreviated as $Y$ in the following) gives a linear map $Y:\Vbb\otimes\Wbb\rightarrow \Wbb((z))$ sending $v\otimes w$ to $Y(v,z)w$. We will write $Y_\Wbb$ as $Y$ when the context is clear. By identifying $\Vbb$ with $\Vbb\otimes 1$ in $\Vbb\otimes\Cbb((z))$ and similarly $\Wbb$ with $\Wbb\otimes 1$ in $\Wbb\otimes\Cbb((z))$, $Y$ can be extended $\Cbb((z))$-bilinearly to \index{Y@$Y_\Wbb=Y$}
\begin{gather}
\begin{array}{c}
Y:\Big(\Vbb\otimes\Cbb((z))\Big)\otimes\Big(\Wbb\otimes \Cbb((z))\Big)\rightarrow \Wbb\otimes\Cbb((z)),\\[1.5ex]
Y(u\otimes f,z)w\otimes g=f(z)g(z)Y(u,z)w
\end{array}\label{eq3}
\end{gather}
(for each $u\in\Vbb,w\in\Wbb,f,g\in\Cbb((z))$). It can furthermore be extended to 
\begin{align}
Y:	\Big(\Vbb\otimes\Cbb((z))dz\Big)\otimes \Big(\Wbb\otimes\Cbb((z))\Big)\rightarrow \Wbb\otimes \Cbb((z))dz \label{eq4}
\end{align}
in an obvious way. Thus, for each $v\in\Vbb\otimes\Cbb((z))dz$, we can define the residue \index{Res@$\Res$}
\begin{align}
\Res_{z=0}~Y(v,z)w,\label{eq5}
\end{align}
which, in case $v=u\otimes fdz,w=m\otimes g$ where $u\in\Vbb$, $m\in\Wbb$,  and $f,g\in\Cbb((z))$, is the $\Wbb$-coefficient of  $f(z)g(z)Y(v,z)mdz$ before $z^{-1}dz$.

We define a group $\Gbb=\{f\in\scr O_{\Cbb,0}:f(0)=0,f'(0)\neq 0\}$ \index{G@$\Gbb$} where the stalk $\scr O_{\Cbb,0}$ is the set of holomorphic functions defined on a neighborhood of $0$. The multiplication  rule of $\Gbb$ is the composition $\rho_1\circ\rho_2$ of any two elements $\rho_1,\rho_2\in\Gbb$. By \cite{Hua97}, for each $\Vbb$-module $\Wbb$, there is a homomorphism $\mc U:\Gbb\rightarrow\Wbb$ defined in the following way: If we choose the unique $c_0,c_1,c_2\dots\in\Cbb$ satisfying
\begin{align*}
\rho(z)=c_0\cdot\exp\Big(\sum_{n>0}c_nz^{n+1}\partial_z\Big)z
\end{align*}
then we necessarily have $c_0=\rho'(0)$, and we set \index{U@$\mc U(\rho)$}
\begin{align*}
\mc U(\rho)=\rho'(0)^{\wtd L_0}\cdot\exp\Big(\sum_{n>0}c_nL_n\Big).
\end{align*}
Note that although the expression of $\mc U(\rho)$ involves infinite series, its restriction to each $\Wbb^{\leq k}$ is a finite sum, because each $\sum_{n>0} c_nL_n$ lowers the $\wtd L_0$-weights by at least $1$ and is therefore nilpotent and equals $\sum_{n=1}^k c_nL_n$ on $\Wbb^{\leq k}$.

If $X$ is a complex manifold,  a (holomorphic) \textbf{family of transformations} $\rho:X\rightarrow\Gbb$ is by definition an analytic function $\rho=\rho(x,z)=\rho_x(z)$ on a neighborhood of  $X\times\{0\}\subset X\times\Cbb$. Then $\mc U(\rho)$ (on each $\Wbb$) is defined pointwisely, which is an $\End(\Wbb)$-valued function on $X$ whose value at each $x\in X$ is $\mc U(\rho_x)$. $\mc U(\rho)$ can be regarded as an $\scr O_X$-module automorphism of $\Wbb\otimes_\Cbb\scr O_X$.

Let $\fk X=(\pi:\mc C\rightarrow\mc B)$ be a family of compact Riemann surfaces. Associated to $\fk X$ one can define a sheaf of $\scr O_X$-modules $\scr V_{\fk X}$ as follows. (Cf. \cite[Chapter 6, 17]{FB04}; our presentation follows \cite[Sec. 5]{Gui23}.) First, suppose $U,V\subset\mc C$ are open subsets, and we have two holomorphic functions $\eta\in\scr O(U),\mu\in\scr O(V)$ locally injective (i.e., \'etale) on  each fiber $U_b:=U\cap\pi^{-1}(b),V_b=V\cap\pi^{-1}(b)$ ($b\in\mc B$) of $U$ and $V$ respectively. We can define a family of transformations $\varrho(\eta|\mu):U\cap V\rightarrow\Gbb$  as follows: for each $p\in\mc C$, both $\eta-\eta(p)$ and $\mu-\mu(p)$ restricts to an injective holomorphic function on the fiber $(U\cap V)_{\pi(p)}=U\cap V\cap\pi^{-1}(\pi(p))$ vanishing at $p$. Then $\varrho(\eta|\mu)_p\in\Gbb$ is determined by \index{zz@$\varrho(\eta\vert\mu)$}
\begin{align}
\boxed{~\eta-\eta(p)\big|_{(U\cap V)_{\pi(p)}}	=\varrho(\eta|\mu)_p\big(\mu-\mu(p)\big|_{(U\cap V)_{\pi(p)}}\big)~}\label{eq10}
\end{align}
on a neighborhood of $0\in\Cbb$. Then $\mc U(\varrho(\eta|\mu))$ is an $\scr O_{U\cap V}$-module automorphism of $\Vbb\otimes_\Cbb\scr O_{U\cap V}$ which restricts to an automorphism of $\Vbb^{\leq n}\otimes_\Cbb\scr O_{U\cap V}$ for each $n\in\Nbb$. The cocycle condition $\varrho(\eta|\mu)\varrho(\mu|\nu)=\varrho(\eta|\nu)$ holds for any holomorphic function $\nu$ on a neighborhood of $\mc C$ which is injective on each fiber. 

Thus, we can define $\scr V_{\fk X}^{\leq n}$  to be the holomorphic vector bundle on $\mc C$ which associates to each open $U\subset\mc C$ and each $\eta\in\scr O(U)$ locally injective on fibers a trivialization (i.e., an isomorphism of $\scr O_U$-modules) \index{U@$\mc U_\varrho(\eta)$}
\begin{align}
\mc U_\varrho(\eta):\scr V_{\fk X}^{\leq n}|_U\xrightarrow{\simeq}\Vbb^{\leq n}\otimes_\Cbb\scr O_U	\label{eq1}
\end{align}
such that for another similar $V\subset\mc C,\mu\in\scr O(V)$, we have the transition function
\begin{align}
\mc U_\varrho(\eta)\mc U_\varrho(\mu)^{-1}=\mc U(\varrho(\eta|\mu)):\Vbb^{\leq n}\otimes_\Cbb\scr O_{U\cap V}\xrightarrow{\simeq}\Vbb^{\leq n}\otimes_\Cbb\scr O_{U\cap V}. \label{eq12}
\end{align} 
If $n'>n$, we have clearly an  $\scr O_{\mc C}$-module monomorphism $\scr V_{\fk X}^{\leq n}\rightarrow\scr V_{\fk X}^{\leq n'}$ which, for each open $U\subset\mc C$ and $\eta$ as above, is transported under the isomorphisms \eqref{eq1} to the canonical monomorphism $\Vbb^{\leq n}\otimes_\Cbb\scr O_U\rightarrow\Vbb^{\leq n'}\otimes_\Cbb\scr O_U$ defined by the inclusion $\Vbb^{\leq n}\hookrightarrow\Vbb^{\leq n'}$. Thus we are allowed to define
\begin{align*}
\scr V_{\fk X}=\varinjlim_{n\in\Nbb}\scr V_{\fk X}^{\leq n}.
\end{align*}
Alternatively, one can directly define $\scr V_{\fk X}$ to be the $\scr O_{\mc C}$-module which is locally free (of infinite rank) and isomorphic to $\Vbb\otimes_\Cbb\scr O_U$ via a morphism $\mc U_\varrho(\eta)$, and whose transition function is given by $\mc U(\varrho(\eta|\mu))$. We call $\scr V_{\fk X}$ the \textbf{sheaf of VOA} associated to $\fk X$ and $\Vbb$. If $\fk X$ is a single compact Riemann surface $C$, we write $\scr V_{\fk X}$ as $\scr V_C$. \index{V@$\scr V_{\fk X}^{\leq n},\scr V_{\fk X},\scr V_C$}

For each fiber $\mc C_b$ (where $b\in\mc B$), we have a canonical equivalence
\begin{align}
\scr V_{\fk X}|_{\mc C_b}\simeq\scr V_{\mc C_b}\equiv\scr V_{\fk X_b}\label{eq15}
\end{align}
such that if these two $\scr O_{\mc C_b}$-modules are identified by this isomorphism, then  the restriction of the trivialization \eqref{eq1} to $U_b=U\cap\pi^{-1}(b)$ equals
\begin{align*}
\mc U_\varrho(\eta|_{\mc C_b}):\scr V_{\mc C_b}|_{U_b}\xrightarrow{\simeq}\Vbb\otimes_\Cbb\scr O_{U_b}.
\end{align*}

\begin{df}\label{lb11}
Since the vacuum vector $\id$ is killed by all $L_n$ (where $n\geq 0$), it is fixed by any change of coordinate $\mc U(\rho)$. It follows that we can define a section $\id\in\scr V_{\fk X}(\mc C)$ which under any trivialization $\mc U_\varrho(\eta)$ is the constant section $\id$, called the \textbf{vacuum section}. \index{1@The vacuum section $\id$}
\end{df}

\section{Conformal blocks}

Let $\fk X$ be a family of $N$-pointed compact Riemann surfaces as in \eqref{eq2}. We choose $\Vbb$-modules $\Wbb_1,\dots,\Wbb_N$. Set \index{W@$\Wbb_\blt,w_\blt=w_1\otimes\cdots\otimes w_N$}
\begin{align*}
\Wbb_\blt=\Wbb_1\otimes\cdots\otimes\Wbb_N.	
\end{align*}
$w\in\Wbb_\blt$ means a vector in $\Wbb_\blt$, and $w_\blt\in\Wbb_\blt$ means a vector of the form $w_1\otimes\cdots\otimes w_N$ where each $w_i\in \Wbb_i$.

The sheaf of conformal blocks is an $\scr O_{\mc B}$-submodule of an infinite-rank locally free $\scr O_{\mc B}$-module $\scr W_{\fk X}(\Wbb_\blt)$, where the latter is defined as follows. For each open subset $V\subset\mc B$ such that the restricted family \index{XV@$\fk X_V,\mc C_V$}
\begin{align*}
\fk X_V:=(\pi:\mc C_V\rightarrow V;\sgm_1|_V,\dots,\sgm_N|_V)	
\end{align*}
(where $\mc C_V=\pi^{-1}(V)$) admits local coordinates $\eta_1,\dots,\eta_N$ at $\sgm_1(V),\dots,\sgm_N(V)$ respectively,  we have a trivialization (i.e., an isomorphism of $\scr O_V$-modules) \index{U@$\mc U(\eta_\blt)$} \index{WXW@$\scr W_{\fk X}(\Wbb_\blt)$}
\begin{align*}
\mc U(\eta_\blt)\equiv\mc U(\eta_1)\otimes\cdots\otimes\mc U(\eta_N):\scr W_{\fk X}(\Wbb_\blt)|_V\xrightarrow{\simeq} \Wbb_\blt\otimes_\Cbb\scr O_V.	
\end{align*}
If $V$ is small enough such that we have another set of local coordinates $\mu_1,\dots,\mu_N$ at $\sgm_1(V),\dots,\sgm_N(V)$ respectively, for each $1\leq j\leq N$ we choose a family of transformations $(\eta_j|\mu_j):V\rightarrow\Gbb$ defined by \index{zz@$(\eta_j\vert\mu_j)$}
\begin{align}
\boxed{~(\eta_j|\mu_j)_b\circ\mu_j|_{\mc C_b}=\eta_j|_{\mc C_b}~}	\label{eq8}
\end{align}
for each $b\in V$. Then each $\mc U(\eta_j|\mu_j)$ is a holomorphic family of invertible endomorphisms of $\Wbb_j$ associated to $(\eta_j|\mu_j)$ (as defined in Sec. \ref{lb1}). The tensor product of them, as a family of invertible transformations of $\Wbb_\blt$ (more precisely, an automorphism of the $\scr O_V$-module $\Wbb_\blt\otimes_\Cbb\scr O_V$), is the transition function: 
\begin{align}
\mc U(\eta_\blt)\mc U(\mu_\blt)^{-1}:=\mc U(\eta_1|\mu_1)\otimes\cdots\otimes\mc U(\eta_N|\mu_N).\label{eq11}
\end{align}
This gives the definition of $\scr W_{\fk X}(\Wbb_\blt)$.

In particular, $\scr W_{\fk X_b}(\Wbb_\blt)$ is a vector space equivalent to $\Wbb_\blt$ through $\mc U(\eta_\blt|_{\mc C_b})$. It is easy to see that for each $b\in\mc B$, the restriction $\scr W_{\fk X}(\Wbb_\blt)|_b$ (i.e., the fiber of the vector bundle at $b$) is naturally equivalent to $\scr W_{\fk X_b}(\Wbb_\blt)$:
\begin{align}
\scr W_{\fk X}(\Wbb_\blt)|_b\simeq	 \scr W_{\fk X_b}(\Wbb_\blt).\label{eq6}
\end{align}
This equivalence is uniquely determined by the fact that if we identify the two spaces, then  the restriction of $\mc U(\eta_\blt)$ to the map $\scr W_{\fk X_b}(\Wbb_\blt)\rightarrow\Wbb_\blt$ equals $\mc U(\eta_\blt|_{\mc C_b})$. 


To define conformal blocks, we first consider the case that $\mc B$ is a single point. Then $C:=\mc C$ is a compact Riemann surface. We can define a linear action of $H^0(C,\scr V_C\otimes\omega_C(\star\SX))$ on $\scr W_{\fk X}(\Wbb_\blt)$ as follows. Choose any local coordinate $\eta_i$ of $C$ at the point $x_j:=\sgm_j(\mc B)$, defined on a neighboorhood $W_j$ of $x_j$ (so, in particular, $\eta_j(x_j)=0$). Note $\SX=\{x_1,\dots,x_N\}$. We assume
\begin{align*}
W_j\cap\SX=\{x_j\}.	
\end{align*}
Note that we have a trivialization
\begin{align*}
\mc U_\varrho(\eta_j):\scr V_C|_{W_i}\xrightarrow{\simeq} \Vbb\otimes_\Cbb\scr O_{W_i}\simeq	\Vbb\otimes_\Cbb\scr O_{\eta_i(W_i)}
\end{align*} 
which, tensored by $(\eta_j^{-1})^*:\omega_{W_j}\xrightarrow{\simeq}\omega_{\eta_j(W_j)}$, gives a trivialization \index{V@$\mc V_\varrho(\eta_j)$}
\begin{align*}
\mc V_\varrho(\eta_j):\scr V_C|_{W_i}\otimes\omega_{W_i}(\star\SX)\xrightarrow{\simeq} \Vbb\otimes_\Cbb\omega_{\eta_j(W_i)}(\star 0)	
\end{align*} 
Then for each $v\in H^0(C,\scr V_C\otimes\omega_C(\star\SX))$, we have a section $\mc V_\varrho(\eta_j)v$, which is a $\Vbb$-valued (more precisely, $\Vbb^{\leq n}$-valued for some $n\in\Nbb$) holomorphic $1$-form on $\eta_j(W_j)$ but possibly has poles at $\eta_j(x_j)=0$. By taking Laurent series expansions, $\mc V_\varrho(\eta_j)v$ can be regarded as an element of $\Vbb\otimes\Cbb((z))dz$. We then define, (notice that we have an isomorphism $\mc U(\eta_\blt):\scr W_{\fk X}(\Wbb_\blt)\xrightarrow{\simeq}\Wbb_\blt$) an action of $v$ on $\scr W_{\fk X}(\Wbb_\blt)$ by
\begin{subequations}
\begin{gather}
\mc U(\eta_\blt)\cdot v\cdot \mc U(\eta_\blt)^{-1}w_\blt=\sum_{j=1}^N w_1\otimes\cdots\otimes \mc U(\eta_\blt)\cdot v\cdot \mc U(\eta_\blt)^{-1}w_j\otimes\cdots\otimes w_N\\
\mc U(\eta_\blt)\cdot v\cdot \mc U(\eta_\blt)^{-1}w_j=\Res_{z=0}~Y\big(\mc V_\varrho(\eta_j)v,z\big)w_j\label{eq63}
\end{gather}
\end{subequations}
for each $w_\blt\in\Wbb_\blt$, where the residue is defined as in \eqref{eq5}. That this definition is independent of the choice of local coordinates $\eta_\blt$ follows from \cite[Thm. 6.5.4]{FB04} (see also \cite[Thm. 3.2]{Gui23}), which relies on a crucial change of variable formula (cf. \cite[Thm. 3.3]{Gui23}) proved by Huang \cite{Hua97}.

Now that we have a linear action of $H^0(C,\scr V_C\otimes\omega_C(\star\SX))$ on $\scr W_{\fk X}(\Wbb_\blt)$, we say that a linear functional $\upphi:\scr W_{\fk X}(\Wbb_\blt)\rightarrow\Cbb$ is a \textbf{conformal block} (associated to $\fk X$ and $\Wbb_\blt$) exactly when $\upphi$ vanishes on the vector space
\begin{align*}
\scr J:=H^0(C,\scr V_C\otimes\omega_C(\star\SX))\cdot \scr W_{\fk X}(\Wbb_\blt)
\end{align*}
where $\Span_\Cbb$ is suppressed on the right hand side of the equality. 

Now we come back to the general setting that $\fk X$ is a family of $N$-pointed compact Riemann surfaces. Let $\upphi:\scr W_{\fk X}(\Wbb_\blt)\rightarrow\scr O_{\mc B}$ be an $\scr O_{\mc B}$-module morphism, which can be understood in the following way: If locally we identify $\scr W_{\fk X}(\Wbb_\blt)|_V$ (where $V$ is an open subset of $\mc B$) with $\Wbb_\blt\otimes_\Cbb\scr O_V$, then $\upphi$ associates to each vector $w\in\Wbb_\blt$ (considered as the constant section $w\otimes 1\in\Wbb_\blt\otimes\scr O(V)$) a holomorphic function $\upphi(w)$ on $U$.

\begin{df}
Let $\upphi:\scr W_{\fk X}(\Wbb_\blt)\rightarrow\scr O_{\mc B}$ be an $\scr O_{\mc B}$-module morphism. For each $b\in\mc B$,  regard $\upphi|_b$ as the restriction of $\upphi$ to the fiber map $\scr W_{\fk X}(\Wbb_\blt)|_{b}\simeq\scr W_{\fk X_b}(\Wbb_\blt)\rightarrow\Cbb$. Then, we say $\upphi$  is  a \textbf{conformal block} (over $\mc B$ associated to $\fk X$ and $\Wbb_\blt$) if for each $b\in\mc B$, $\upphi|_b$ is a  conformal block associated to $\fk X_b$ (i.e., $\upphi(b)$ vanishes on $H^0(\mc C_b,\scr V_{\mc C_b}\otimes\omega_{\mc C_b}(\star\SX|_b))\cdot\scr W_{\fk X_b}(\Wbb_\blt)$).
\end{df}

The following proposition is \cite[Prop. 6.4]{Gui23}.

\begin{pp}\label{lb10}
Let $\upphi:\scr W_{\fk X}(\Wbb_\blt)\rightarrow\scr O_{\mc B}$ be an $\scr O_{\mc B}$-module morphism. Suppose that  each connected component of $\mc B$ contains a non-empty open subset $V$ such that the restriction of $\upphi$ to $\scr W_{\fk X_V}(\Wbb_\blt)\rightarrow\scr O_V$ is a conformal block, then the original $\upphi$ is a conformal block associated to $\fk X$ and $\Wbb_\blt$.
\end{pp}

\section{Sewing conformal blocks}\label{lb18}

Let $N,M\in\Zbb_+$. Let
\begin{align*}
\wtd{\fk X}=(\wtd\pi:\wtd{\mc C}\rightarrow\wtd{\mc B};\sgm_1,\dots,\sgm_N;\sgm_1',\dots,\sgm_M';\sgm_1'',\dots,\sgm_M'';)
\end{align*}
be a family of $(N+2M)$-pointed compact Riemann surfaces. Unless otherwise stated, we assume the following condition.
\begin{ass}\label{lb26}
We assume that for every $\wtd b\in\wtd{\mc B}$, each connected component of the fiber $\wtd{\mc C}_{\wtd b}$ contains one of $\sgm_1(\wtd b),\dots,\sgm_N(\wtd b)$.  
\end{ass}
For each $1\leq j\leq M$,  we assume $\wtd{\fk X}$ has  local coordinates $\xi_j$ at $\sgm_j'(\wtd{\mc B})$ defined on a neighborhood $W_j'\subset\wtd{\mc C}$ of $\sgm_j'(\wtd{\mc B})$ and similarly $\varpi_j$ at $\sgm_j''(\wtd{\mc B})$ defined on a neighborhood $W_j''$. We assume all $W_j',W_j''$ ($1\leq j\leq M$) are mutually disjoint and are also disjoint from $\sgm_1(\wtd{\mc B}),\dots,\sgm_N(\wtd {\mc B})$, so that $\sgm_1(\wtd{\mc B}),\dots,\sgm_N(\wtd {\mc B})$ remain after sewing. We also assume that for each $1\leq j\leq M$, we can choose $r_j,\rho_j>0$  such that
\begin{gather}
	(\xi_j,\wtd\pi):W_j'\xrightarrow{\simeq} \mc D_{r_j}\times\wtd{\mc B}\qquad\text{resp.}\qquad (\varpi_j,\wtd\pi):W_j''\xrightarrow{\simeq} \mc D_{\rho_j}\times\wtd{\mc B}\label{eq94}
\end{gather}
is a biholomorphic map. (Recall that $\mc D_r$ is the open disc at $0\in\Cbb$ with radius $r$.) 

We do not assume $\wtd{\fk X}$ has local coordinates at $\sgm_1(\wtd{\mc B}),\dots,\sgm_N(\wtd{\mc B})$.

\subsection*{Sewing families of compact Riemann surfaces}

We can \textbf{sew $\wtd{\fk X}$ along all pairs $\sgm_j'(\wtd{\mc B}),\sgm_j''(\wtd{\mc B})$} to obtain a new family
\begin{align}
\fk X=(\pi:\mc C\rightarrow\mc B;\sgm_1,\dots,\sgm_N)	
\end{align}
of compact Riemann surfaces. Here, \index{Drrho@$\mc D_{r_\blt\rho_\blt}^\times$}
\begin{gather*}
\mc B=	\mc D^\times_{r_\blt\rho_\blt}\times\wtd{\mc B},\qquad \mc D^\times_{r_\blt\rho_\blt}=\mc D^\times_{r_1\rho_1}\times\cdots\times\mc D^\times_{r_M\rho_M}.
\end{gather*}
$\fk X$ is described as follows.

For each $q_\blt\in\mc D_{r_\blt\rho_\blt}^\times$ and $b\in\wtd{\mc B}$, the fiber $\mc C_{(q_\blt,b)}$ is obtained by  removing the closed discs 
\begin{align*}
F_{j,b}'=\{y\in W_j'\cap\wtd{\mc C}_b:|\xi_j(y)|\leq |q_j|/\rho_j\},\qquad F_{j,b}''=\{y\in W_j''\cap\wtd{\mc C}_b:|\varpi_j(y)|\leq |q_j|/r_j\}	
\end{align*}
(for all $j$) from $\wtd{\mc C}_b$, and gluing the remaining part of the Riemann surface $\wtd {\mc C}_b$ by identifying (for all $j$) $y'\in W_j'\cap \wtd{\mc C}_b$ with $y''\in W_j''\cap\wtd{\mc C}_b$ if $\xi_j(y')\varpi_j(y'')=q_j$. This procedure can be performed in a consistent way over all $b\in\wtd{\mc B}$, which gives $\pi:\mc C\rightarrow\mc B$. See for instance \cite[Sec. 4]{Gui23} for details.\footnote{Indeed,  one can extend $\fk X$ to a slightly larger flat family of complex curves (with at worst nodal singularities) with base manifold $\mc D_{r_\blt\rho_\blt}\times\wtd{\mc B}$ (cf. for instance \cite[Sec. 4]{Gui23}).} 

Since $\Omega=\wtd{\mc C}\setminus \bigcup_j(W_j'\cup W_j'')$ is not affected by gluing, $\mc D_{r_\blt\rho_\blt}^\times\times\Omega$ can be viewed as a subset of $\fk X$, and the restriction of $\pi$ to this set is $\mc D_{r_\blt\rho_\blt}^\times\times\Omega\xrightarrow{1\otimes\wtd\pi}\mc D_{r_\blt\rho_\blt}^\times\times\wtd{\mc B}=\mc B$. Thus, for each $1\leq i\leq N$ the section $\sgm_i$ for $\wtd{\fk X}$  defines the corresponding section $1\times\sgm_i:\mc D_{r_\blt\rho_\blt}^\times\times\wtd{\mc B}\rightarrow \mc D_{r_\blt\rho_\blt}^\times\times\Omega$, also denoted by $\sgm_i$. A local coordinate $\eta_i$ of $\wtd{\fk X}$ at $\sgm_i(\wtd{\mc B})$ extends constantly over $\mc D_{r_\blt\rho_\blt}^\times$ to a local coordinate of $\fk X$ at $\sgm_i(\mc B)$, also denoted by $\eta_i$.

\subsection*{Sewing conformal blocks}

We now define sewing conformal blocks associated to $\wtd{\fk X}$. Associate to $\sgm_1,\dots,\sgm_N$  $\Vbb$-modules $\Wbb_1,\dots,\Wbb_N$. Then we have $\scr W_{\wtd{\fk X}}(\Wbb_\blt)$ defined by $(\wtd\pi:\wtd{\mc C}\rightarrow\wtd{\mc B};\sgm_1,\dots,\sgm_N)$. For each connected open $\wtd V\subset\wtd{\mc B}$, $\scr W_{\wtd{\fk X}}(\Wbb_\blt)(\wtd V)$ can be identified canonically with a subspace of $\scr W_{\fk X}(\Wbb_\blt)(\mc D_{r_\blt\rho_\blt}^\times\times\wtd V)$ consisting of sections of the latter which are constant with respect to sewing. More precisely, this identification is compatible with restrictions to open subsets of $\wtd V$; moreover, if $\wtd V$ is small enough such that $\wtd{\fk X}|_{\wtd V}$ has local coordinates $\eta_1,\dots,\eta_N$ at $\sgm_1(\wtd V),\dots,\sgm_N(\wtd V)$ which give rise to $\eta_1,\dots,\eta_N$ of $\eta_1,\dots,\eta_N$ of $\fk X|_{\mc D_{r_\blt\rho_\blt}^\times\times\wtd V}$ at $\sgm_1(\mc D_{r_\blt\rho_\blt}^\times\times\wtd V),\dots,\sgm_N(\mc D_{r_\blt\rho_\blt}^\times\times\wtd V)$ (which are constant over $\mc D_{r_\blt\rho_\blt}^\times$), then the following diagram commutes:
\begin{equation}\label{eq28}
\begin{tikzcd}
\scr W_{\wtd{\fk X}}(\Wbb_\blt)(\wtd V) \arrow[r,hook] \arrow[dd,"\simeq"',"\mc U(\eta_\blt)"] &\scr W_{\fk X}(\Wbb_\blt)(\mc D_{r_\blt\rho_\blt}^\times\times\wtd V)\arrow[dd,"\simeq","\mc U(\eta_\blt)"']\\\\
\Wbb_\blt\otimes_\Cbb\scr O(\wtd V) \arrow[r,hook]&\Wbb_\blt\otimes_\Cbb\scr O(\mc D_{r_\blt\rho_\blt}^\times\times\wtd V)
\end{tikzcd}	
\end{equation}
where the bottom horizontal line is defined by pulling pack the projection $\mc D_{r_\blt\rho_\blt}^\times\times\wtd V\rightarrow\wtd V$.

Associate to $\sgm_1',\dots,\sgm_M'$  $\Vbb$-modules $\Mbb_1,\dots,\Mbb_M$, whose contragredient modules $\Mbb_1',\dots,\Mbb_M'$ are associated to $\sgm_1'',\dots,\sgm_M''$. We understand $\Wbb_\blt\otimes\Mbb_\blt\otimes\Mbb_\blt'$ as
\begin{gather*}
\Wbb_1\otimes\cdots\otimes \Wbb_N\otimes\Mbb_1\otimes\Mbb_1'\otimes\cdots\otimes\Mbb_M\otimes\Mbb_M',
\end{gather*}
where the order has be changed so that each $\Mbb_j'$ is next to $\Mbb_j$. We can then identify
\begin{align}\label{eq27}
\scr W_{\wtd{\fk X}}(\Wbb_\blt\otimes\Mbb_\blt\otimes\Mbb_\blt')=\scr W_{\wtd{\fk X}}(\Wbb_\blt)\otimes_\Cbb\Mbb_\blt\otimes\Mbb_\blt'	
\end{align}
such that whenever $\wtd V\subset\wtd{\mc B}$ is open such that $\wtd{\fk X}|_{\wtd V}$ has local coordinates $\eta_1,\dots,\eta_N$ at $\sgm_1(\wtd V),\dots,\sgm_N(\wtd V)$ as before, the following diagram commutes:
\begin{equation}
\begin{tikzcd}
\scr W_{\wtd{\fk X}}(\Wbb_\blt\otimes\Mbb_\blt\otimes\Mbb_\blt')\big|_{\wtd V} \arrow[rr,leftrightarrow,"="] \arrow[ddr,"\simeq"',"{\mc U(\eta_\blt,\xi_\blt,\varpi_\blt)}"] &&\scr W_{\wtd{\fk X}}(\Wbb_\blt)\big|_{\wtd V}\otimes_\Cbb\Mbb_\blt\otimes\Mbb_\blt'\arrow[ddl,"\simeq","\mc U(\eta_\blt)\otimes\id"']\\{}\\
&\Wbb_\blt\otimes\Mbb_\blt\otimes\Mbb_\blt'\otimes_\Cbb\scr O_{\wtd V} 
\end{tikzcd}	
\end{equation}

We define
\begin{gather*}
q_j^{\wtd L_0}\btr\otimes_j\btl=\sum_{n\in\Nbb} q_j^n\sum_{a\in\fk A_{j,n}} m{(n,a)}\otimes\wch m{(n,a)}\qquad \in(\Mbb_j\otimes\Mbb_j')[[q_j]]
\end{gather*}
where for each $n\in\Nbb,s\in\Cbb$, $\{m{(n,a)}:a\in\fk A_{j,n}\}$ is a basis of $\Wbb{(n)}$ with dual basis	$\{\wch m{(n,a)}:a\in\fk A_{j,n}\}$ in $\Wbb'{(n)}$.



Now, for any conformal block $\uppsi:\scr W_{\wtd{\fk X}}(\Wbb_\blt\otimes\Mbb_\blt\otimes\Mbb_\blt')\rightarrow\scr O_{\wtd{\mc B}}$ associated to the family $\wtd{\fk X}$ and $\Wbb_\blt\otimes\Mbb_\blt\otimes\Mbb_\blt'$, we define an $\scr O_{\wtd{\mc B}}$-module morphism \index{S@$\wtd{\mc S}\uppsi$}
\begin{gather*}
\wtd{\mc S}\uppsi:\scr W_{\wtd{\fk X}}(\Wbb_\blt)\rightarrow\scr O_{\wtd{\mc B}}[[q_1,\dots,q_M]]
\end{gather*}
by sending each section $w$ over an open $\wtd V\subset\wtd{\mc B}$ to 
\begin{align}
\wtd{\mc S}\uppsi(w)=\uppsi\Big(w\otimes (q_1^{\wtd  L_0}\btr\otimes_1\btl)\otimes\cdots\otimes (q_M^{\wtd  L_0}\btr\otimes_M\btl)\Big)\qquad\in\scr O(\wtd V)[[q_1,\dots,q_M]].
\end{align}
The identification \eqref{eq27} is used in this definition. $\wtd{\mc S}\uppsi$ is called the (normalized) \textbf{sewing of $\uppsi$}.

\begin{df}\label{lb19}
Let $X$ be a complex manifold. Consider an element
\begin{align*}
f=\sum_{n_1,\dots,n_M\in\Cbb} f_{n_1,\dots,n_M}q_1^{n_1}\cdots q_M^{n_M}	\qquad\in\scr O(X)\{q_1,\dots,q_M\}
\end{align*}
where each $f_{n_1,\dots,n_M}\in\scr O(X)$. Let $R_1,\dots,R_M\in[0,+\infty]$ and $\mc D_{R_\blt}^\times=\mc D_{R_1}^\times\times\cdots\times\mc D_{R_M}^\times$. For any locally compact subset $\Omega$ of $\mc D_{R_\blt}^\times\times X$, we say that formal series $f$ \textbf{converges absolutely and locally uniformly (a.l.u.)  on $\Omega$}, if for any compact subsets $K\subset\Omega$, we have
\begin{align*}
\sup_{(q_\blt,x)\in K}~\sum_{n_1,\dots,n_M\in\Cbb} \big|f_{n_1,\dots,n_M}(x)q_1^{n_1}\cdots q_M^{n_M}	\big| <+\infty.
\end{align*}  
In the case that $f\in\scr O(X)[[q_1^{\pm 1},\dots,q_M^{\pm 1}]]$, it is clear from complex analysis that $f$ converges a.l.u. on $\mc D_{R_\blt}^\times\times X$ if and only if $f$ is the Laurent series expansion of an element (also denoted by $f$) of $\scr O(\mc D_{R_\blt}^\times\times X)$.
\end{df}

\begin{df}
We say that $\wtd{\mc S}\uppsi$ \textbf{converges a.l.u.} (on $\mc B=\mc D_{r_\blt\rho_\blt}^\times\times\wtd{\mc B}$), if for any open subset $\wtd V\subset\wtd{\mc B}$ and any section $w$ of $\scr W_{\wtd{\fk X}}(\Wbb_\blt)(\wtd V)$, $\wtd{\mc S}\uppsi(w)$  converges a.l.u. on $\mc D_{r_\blt\rho_\blt}^\times\times\wtd V$.
\end{df}


Consider the following condition weaker than assumption \ref{lb26}:
\begin{ass}\label{lb27}
For every $b\in\mc B$, each connected component of the fiber $\mc C_b$ contains one of $\sgm_1(b),\dots,\sgm_N(b)$.
\end{ass}

\begin{thm}[\cite{Gui23}, Thm. 11.3]\label{lb8}
Assume Assumption \ref{lb27} instead of Assumption \ref{lb26}. If $\wtd{\mc S}\uppsi$ converges a.l.u. on  $\mc B=\mc D_{r_\blt\rho_\blt}^\times\times\wtd{\mc B}$, then $\wtd{\mc S}\uppsi$ (resp. $\mc S\uppsi$), when extended $\scr O_{\mc B}$-linearly to an $\scr O_{\mc B}$-module homomorphism $\scr W_{\fk X}(\Wbb_\blt)\rightarrow\scr O_{\mc B}$ using the inclusion $\scr W_{\wtd{\fk X}}(\Wbb_\blt)\subset \scr W_{\fk X}(\Wbb_\blt)$ defined by \eqref{eq28},  is a conformal block associated to $\fk X$ and $\Wbb_\blt$. 
\end{thm}

\begin{eg}\label{lb9}
Let $\fk Y=(C;x_1,\dots,x_N)$ be an $N$-pointed compact Riemann surface with local coordinates $\eta_1,\dots,\eta_N$ at $x_1,\dots,x_N$, defined on neighborhoods $W_1,\dots,W_N$ satisfying $W_j\cap\{x_1,\dots,x_N\}=x_j$ for each $1\leq j\leq N$. Assume $\eta_1(W_1)=\mc D_r$ for some $r>0$. Let $\zeta$ be the standard coordinate of $\Cbb$. Let $\wtd{\fk X}$ be the disjoint union of $\fk Y$ and $(\Pbb^1;0,1,\infty)$, namely, we have an $(N+3)$-pointed compact Riemann surface
\begin{align*}
\wtd{\fk X}=(C\sqcup \Pbb^1;x_1,\dots,x_N,0,1,\infty).	
\end{align*}
We equip $\wtd{\fk X}$ with local coordinates $\eta_1,\dots,\eta_N,\zeta,(\zeta-1),\zeta^{-1}$. The local coordinate $\zeta$ at $0$ should be defined at $|z|<1$ so that no marked points other than $0$ is inside this region.

We sew $\wtd{\fk X}$ along $x_1\in C$ and $\infty\in\Pbb^1$ using the chosen local coordinates $\eta_1$ and $1/\zeta$ to obtain a family $\fk X$. Then
\begin{align*}
\fk X=(\pi: C\times\mc D_r^\times\rightarrow\mc D_r^\times;x_1,x_2,\dots,x_N,\sgm)	
\end{align*}
where $\pi$ is the projection onto the $\mc D_r^\times$-component, the sections $x_1,\dots,x_N$ are (rigorously speaking) sections sending $q$ to $(x_1,q),\dots,(x_N,q)$. The section $\sgm$ is defined by $\sgm(q)=(\eta_1^{-1}(q),q)$, where $\eta_1^{-1}$ sends $\mc D_r$ biholomorphically to $W_1$. Moreover, the local coordinates of $\fk X$ defined naturally by those of $\wtd{\fk X}$ are described as follows: For each $|q|<r$, their restrictions to
\begin{align}
\fk X_q=(C;x_1,x_2,\dots,x_N,\eta_1^{-1}(q))  \label{eq31}	
\end{align}
are $q^{-1}\eta_1,\eta_2,\dots,\eta_N,q^{-1}(\eta_1-q)$. 

Note that Assumption \ref{lb27} is always satisfied, but Assumption \ref{lb26} is not satisfied when $N=1$.

Attach $\Vbb$-modules $\Wbb_1,\dots,\Wbb_N,\Wbb_1,\Vbb,\Wbb_1'$ to the marked points $x_1,\dots,x_N,0,1,\infty$ respectively of $\wtd{\fk X}$. Fix the trivializations of $\scr W$-sheaves using the chosen local coordinates. Let $\upphi:\Wbb_1\otimes\cdots\otimes\Wbb_N\rightarrow\Cbb$ be a conformal block associated to $(C;x_1,\dots,x_N)$ and $\Wbb_1,\dots,\Wbb_N$. Let
\begin{gather*}
\upomega:\Wbb_1\otimes\Vbb\otimes\Wbb_1'\rightarrow\Cbb,	\\
w\otimes u\otimes w'\mapsto \bk{Y(u,1)w,w'}=\sum_{n\in\Zbb}\bk{Y(u)_nw,w'},
\end{gather*}
which is a conformal block associated to $(\Pbb^1;0,1,\infty)$ and $\Wbb_1,\Vbb,\Wbb_1'$. Then $\uppsi:=\upphi\otimes\upomega$ is a conformal block for $\wtd{\fk X}$. Note that when $u,w_1$ are $\wtd L_0$-homogeneous  (i.e. eigenvectors of $\wtd L_0$) with eigenvalues (weights) $\wtd\wt(u),\wtd\wt(w_1)\in\Nbb$ respectively, by \eqref{eq34}, $Y(u)_nw_1$ is $\wtd L_0$-homogeneous with weight $\wtd\wt(u)+\wtd\wt(_1)-n-1$. Then
\begin{gather}
\begin{array}{c}
\wtd{\mc S}\uppsi:\Wbb_1\otimes\cdots\Wbb_N\otimes\Vbb\rightarrow\Cbb[[q]]\\[1.5ex]
\wtd{\mc S}\uppsi(w_1\otimes\cdots\otimes w_N\otimes u)=\sum\limits_{n\in\Zbb} q^{\wtd\wt(u)+\wtd\wt(w_1)-n-1}\cdot\uppsi\big(Y(u)_nw_1\otimes w_2\otimes\cdots\otimes w_N \big)
\end{array}	\label{eq29}
\end{gather}
when $u,w_1$ are $\wtd L_0$-homogeneous. 

From \cite[Sec. 10.1]{FB04}, this series converges a.l.u. on $\mc D_r^\times$ (i.e. when $0<|q|<r$). (See the proof of Thm. \ref{lb4} for the detailed explanation.) Then, by Theorem \ref{lb8}, for each $0<|q|<r$, \eqref{eq29} converges to a conformal block associated to $\fk X_q$ and the local coordinates mentioned after \eqref{eq31}. If we change the coordinates at $x_1$ and $\eta_1^{-1}(q)$ to $\eta_1$ and $\eta_1-q$ respectively, then in the formula \eqref{eq29},  $u$ and $w_1$ should be multiplied both by $q^{-\wtd L_0}$. Under the trivialization given by the new coordinates, $\wtd{\mc S}\uppsi(w_1\otimes\cdots\otimes w_N\otimes u)$ equals
\begin{align}
\uppsi\big(Y(u,q)w_1\otimes w_2\otimes\cdots\otimes w_N \big):=\sum\limits_{n\in\Zbb} q^{-n-1}\cdot\uppsi\big(Y(u)_nw_1\otimes w_2\otimes\cdots\otimes w_N \big).\label{eq30}	
\end{align}
We conclude that (once the a.l.u. convergence is established) for all $0<|q|<r$, \eqref{eq30} is a conformal block associated to $\fk X_q$, local coordinates $\eta_1,\eta_2\dots,\eta_N,\eta_1-q$, and modules $\Wbb_1,\dots,\Wbb_N,\Vbb$.
\end{eg}

\section{An equivalence of sheaves}\label{lb15}

Recall $\wr\fk X=(\wr\pi:\wr\mc C\rightarrow\wr\mc B;\sigma,\wr\sgm_1,\dots,\wr\sgm_N)$ in \eqref{eq7}. In particular, $\wr\mc C=\mc C\times_{\mc B}(\mc C\setminus\SX)$, $\wr\mc B=\mc C\setminus\SX$. The goal of this section is to establish a canonical $\scr O_{\mc C\setminus\SX}$-module isomorphism
\begin{align}
\scr W_{\wr\fk X}(\Vbb\otimes\Wbb_\blt)\simeq	\scr V_{\fk X}\otimes_{\scr O_{\mc C}}\pi^*\scr W_{\fk X}(\Wbb_\blt)|_{\mc C\setminus\SX}, \label{eq62}
\end{align}
which relates the sheaves of VOAs and the $\scr W$-sheaves.

The reason for establishing this equivalence is the following: We want to construct $n$-times propagation $\wr^n\upphi$ of a conformal block $\upphi$ associated to a fixed pointed compact Riemann surface $\fk X_0$ by induction on $n$. $\wr^n\upphi$ is the propagation of $\wr^{n-1}\upphi$ where the latter is viewed as a conformal block associated to the family of compact Riemann surfaces (namely $\wr^{n-1}\fk X_0$, using the notations in Sec. \ref{lb2}) describing the motion of $n-1$ distinct points on $\fk X_0$. To understand $\wr^{n-1}\upphi$ as a conformal block, we need to describe the $\scr W$-sheaf on $\wr^{n-1}\fk X_0$ using sheaves of VOAs. By setting $\wr^{n-2}\fk X_0=\fk X$ and hence $\wr\fk X=\wr^{n-1}\fk X_0$, one needs to describe the $\scr W$-sheaf on $\wr\fk X$. This is fulfilled by the isomorphism \eqref{eq62}.

Let us begin the formal discussion. Note that $\pi^*\scr W_{\fk X}(\Wbb_\blt)$ is the pullback sheaf $\scr W_{\fk X}(\Wbb_\blt)\otimes_{\scr O_{\mc B}}\scr O_{\mc C}$. This is the sheaf for the presheaf associating to each open $U\subset\mc C$ the $\scr O(U)$-module $\scr W_{\fk X}(\Wbb_\blt)\big(\pi(U)\big)\otimes_{\scr O(\pi(U))}\scr O(U)$. (Note that $\pi$ is an open map.)  Assume the restriction $\fk X_{\pi(U)}$ has local coordinates $\eta_1,\dots,\eta_N$ at $\sgm_1(\pi(U)),\dots,\sgm_N(\pi(U))$. We write
\begin{align*}
\pi^* w:=w\otimes 1\in\scr W_{\fk X}(\Wbb_\blt)\otimes_{\mc B}\scr O_{\mc C}=\pi^*\scr W_{\fk X}(\Wbb_\blt)	
\end{align*}
for any section $w\in\scr W_{\fk X}(\Wbb_\blt)$. Sheafifying the tensor product $\mc U(\eta_\blt)\otimes 1$ on the presheaf provides an isomorphism of $\scr O_{\mc C}$-modules \index{Uzz@$\pi^*\mc U(\eta_\blt)$}
\begin{align}
\pi^*\mc U(\eta_\blt)\equiv \mc U(\eta_\blt)\otimes 1:\scr W_{\fk X}(\Wbb_\blt)\big|_U\otimes_{\scr O_{\pi(U)}}\scr O_U\xrightarrow{\simeq}\Big(\Wbb_\blt\otimes_\Cbb\scr O_{\pi(U)}\Big)\otimes_{\scr O_{\pi(U)}}\scr O_U\label{eq13}
\end{align}
or simply a trivialization (i.e. an $\scr O_U$-module isomorphism)
\begin{align}
\pi^*\mc U(\eta_\blt):\pi^*\scr W_{\fk X}(\Wbb_\blt)\big|_U\xrightarrow{\simeq}\Wbb_\blt\otimes_\Cbb\scr O_U.	
\end{align}
Choose $\mu\in\scr O(U)$ injective on each fiber of $U$. Then we have a trivialization
\begin{align}
\boxed{~\mc U_\varrho(\mu)\otimes\pi^*\mc U(\eta_\blt):	\scr V_{\fk X}\otimes\pi^*\scr W_{\fk X}(\Wbb_\blt)\big|_U\xrightarrow{\simeq}\Vbb\otimes\Wbb_\blt\otimes_\Cbb\scr O_U~}\label{eq16}
\end{align}

Now assume $U\subset\mc C\setminus\SX=\wr\mc B$. Then we can equip the family $\wr\fk X_U$ with local coordinates as follows. For the local coordinate at each submanifold $\wr\sgm_j(U)$ of $\wr\mc C_U=\wr\mc C\cap\wr\pi^{-1}(U)$, we choose $\wr\eta_j$ defined by \index{zz@$\wr\eta_j$, $\wr^n\eta_j$}
\begin{gather}
\wr\eta_j(x,y)=\eta_j(x)
\end{gather}
whenever $(x,y)\in\mc C\times_{\mc B}\mc C\setminus\SX$ makes the above definable. The local coordinate at $\sigma(U)$ is $\triangle\mu$ given by  \index{zz@$\triangle\mu$, $\triangle_i\mu$}
\begin{align}
\triangle\mu(x,y)=\mu(x)-\mu(y)\label{eq9}
\end{align}
when $(x,y)\in U\times_{\mc B} U$. (Recall that $\sigma$ is the diagonal map.) We can then use $\triangle\mu,\wr\eta_\blt=(\wr\eta_1,\dots,\wr\eta_N)$ to obtain a trivialization
\begin{align}
\boxed{~\mc U(\triangle\mu,\wr\eta_\blt):\scr W_{\wr\fk X}(\Vbb\otimes\Wbb_\blt)|_U\xrightarrow{\simeq} \Vbb\otimes\Wbb_\blt\otimes_\Cbb\scr O_U~}\label{eq17}
\end{align}

We shall relate the two trivializations. First, we need a lemma. Recall $U\subset\mc C\setminus\SX$. Recall \eqref{eq10} and \eqref{eq8}.

\begin{lm}\label{lb3}
Suppose $\eta'_1,\dots,\eta_N'$ are local coordinates of $\fk X_{\pi(U)}$ at $\sgm_1(\pi(U)),\dots,\sgm_N(\pi(U))$ respectively, and $\mu'\in\scr O(U)$ is injective on each fiber of $U$. Then, for each $x\in U$, we have
\begin{gather*}
(\wr\eta_j|\wr\eta'_j)_x=(\eta_j|\eta_j')_{\pi(x)},\qquad (\triangle\mu|\triangle\mu')_x=\varrho(\mu|\mu')_x.
\end{gather*}
\end{lm}

Note that $(\wr\eta_j|\wr\eta'_j)$ is a family of transformations over $U\subset\wr\mc B=\mc C\setminus\SX$, and the transformation over the point $x$ is $(\wr\eta_j|\wr\eta'_j)_x$. $(\triangle\mu|\triangle\mu')_x$ is understood in a similar way.

\begin{proof}
We identify $\wr\mc C_x$ with $\mc C_{\pi(x)}$ by identifying $(y,x)\in\mc C\times_{\mc B}\mc C\setminus\SX$ with $y\in\mc C_{\pi(x)}$. Then, from the definition of $\wr\eta_j,\wr\eta_j'$, we clearly have $\wr\eta_j|_{\wr\mc C_x}=\eta_j|_{\mc C_{\pi(x)}}$ and $\wr\eta_j'|_{\wr\mc C_x}=\eta_j'|_{\mc C_{\pi(x)}}$. By \eqref{eq8}, we have
\begin{gather*}
(\wr\eta_j|\wr\eta_j')_x\circ\wr\eta_j'|_{\wr\mc C_x}=\wr\eta_j|_{\wr\mc C_x},\\
(\eta_j|\eta_j')_{\pi(x)}\circ\eta_j'|_{\mc C_{\pi(x)}}=\eta_j|_{\mc C_{\pi(x)}}.
\end{gather*}
This proves $(\wr\eta_j|\wr\eta'_j)_x=(\eta_j|\eta_j')_{\pi(x)}$. 

Similarly, 
\begin{gather*}
(\triangle\mu|\triangle\mu')_x\circ\triangle\mu'|_{\wr\mc C_x}=\triangle\mu|_{\wr\mc C_x}.
\end{gather*}
By \eqref{eq9}, we have $\triangle\mu|_{\wr\mc C_x}=(\mu-\mu(x))|_{\mc C_{\pi(x)}}$ and $\triangle\mu'|_{\wr\mc C_x}=(\mu'-\mu'(x))|_{\mc C_{\pi(x)}}$. These imply
\begin{gather*}
(\triangle\mu|\triangle\mu')_x\circ(\mu'-\mu'(x))|_{\mc C_{\pi(x)}}=(\mu-\mu(x))|_{\mc C_{\pi(x)}}.
\end{gather*}
Comparing this relation with \eqref{eq10} shows that $(\triangle\mu|\triangle\mu')_x=\varrho(\mu|\mu')_x$.
\end{proof}

\begin{pp}\label{lb5}
We have a unique isomorphism of $\scr O_{\mc C\setminus\SX}$-modules (i.e. a unique isomorphism of holomorphic vector bundles on $\mc C\setminus\SX$)
\begin{align}
\Psi_{\fk X}:\scr W_{\wr\fk X}(\Vbb\otimes\Wbb_\blt)\xrightarrow{\simeq} \scr V_{\fk X}\otimes_{\scr O_{\mc C}}\pi^*\scr W_{\fk X}(\Wbb_\blt)\big|_{\mc C\setminus\SX}\label{eq32}	
\end{align}
such that for any open $U\subset\mc C\setminus\SX$ and $\mu,\triangle\mu,\wr\eta_\blt$ as above, the restriction of this isomorphism to $U$ makes the following diagram commutes.
\begin{equation}
\begin{tikzcd}
\scr W_{\wr\fk X}(\Vbb\otimes\Wbb_\blt)\big|_U \arrow[rr,"\simeq"',"\Psi_{\fk X}"] \arrow[ddr, "\simeq"' ,"{\mc U(\triangle\mu,\wr\eta_\blt)}"]  && \scr V_{\fk X}\otimes\pi^*\scr W_{\fk X}(\Wbb_\blt)\big |_U \arrow[ddl,"\simeq","{\mc U_\varrho(\mu)\otimes\pi^*\mc U(\eta_\blt)}"']\\
{}\\
& \Vbb\otimes\Wbb_\blt\otimes_\Cbb \scr O_U
\end{tikzcd}\label{eq26}	
\end{equation}
\end{pp}

\begin{proof}
One can define an isomorphism $\Psi_{\fk X}$ such that the above diagram commutes. Such isomorphism is clearly unique. Thus, it remains to check that $\Psi_{\fk X}$ is well-defined. We will do so by checking that the transition functions of the two sheaves agree.

Assume $U$ is small enough such that we can have another set of $\mu',\eta_\blt'$ similar to $\mu,\eta_\blt$. Then by \eqref{eq11} and Lemma \ref{lb3}, for each $x\in U$, we have equalities
\begin{align}
&\mc U(\triangle\mu,\wr\eta_\blt)_x\cdot \mc U(\triangle\mu',\wr\eta_\blt')_x^{-1}=\mc U(\triangle\mu|\triangle\mu')_x\otimes\mc U(\wr\eta_1|\wr\eta_1')_x\otimes\cdots\otimes \mc U(\wr\eta_N|\wr\eta_N')_x \nonumber \\
=&\mc U(\varrho(\mu|\mu')_x)\otimes \mc U(\eta_1|\eta_1')_{\pi(x)}\otimes\cdots\otimes \mc U(\eta_N|\eta_N')_{\pi(x)}\label{eq14}
\end{align}
for transformations on $\Vbb\otimes\Wbb_\blt\otimes_\Cbb\scr O_U|_x\simeq \Vbb\otimes\Wbb_\blt$. 

By \eqref{eq11} and \eqref{eq13},  we have 
\begin{align}
	\big(\pi^*\mc U(\eta_\blt)\big)_x\cdot \big(\pi^*\mc U(\eta_\blt')\big)^{-1}_x=\mc U(\eta_1|\eta_1')_{\pi(x)}\otimes\cdots\otimes \mc U(\eta_N|\eta_N')_{\pi(x)}	
\end{align}
for automorphisms of $\Wbb_\blt\otimes_\Cbb\scr O_U|_x\simeq\Wbb_\blt$. Thus, by \eqref{eq11} and \eqref{eq12},
\begin{align*}
&\big(\mc U_\varrho(\mu)\otimes\pi^*\mc U(\eta_\blt) \big)_x\cdot\big(\mc U_\varrho(\mu')\otimes\pi^*\mc U(\eta_\blt')\big)_x^{-1}\\
=&	\mc U(\varrho(\mu|\mu'))_x\otimes \mc U(\eta_1|\eta_1')_{\pi(x)}\otimes\cdots\otimes \mc U(\eta_N|\eta_N')_{\pi(x)},
\end{align*}
which equals \eqref{eq14}.
\end{proof}

\section{Propagation of conformal blocks}\label{lb22}

The main result of this section, Thm. \ref{lb4}, says that any conformal block $\upphi$ associated to a family of pointed compact Riemann surfaces $\fk X$ has a propagation $\wr\upphi$, which is a conformal block associated to $\wr\fk X$. Recall that, intuitively, $\wr\fk X$ is the family describing the motion of a point on $\fk X$
not meeting the marked points of $\fk X$. A crucial consequence of Thm. \ref{lb4} (reflected by the fact that $\wr\upphi$ is an $\scr O_{\mc C\setminus \SX}$-module) is that $\wr\upphi$ \textit{is simultaneously holomorphic with respect to the  parameter of the base manifold of $\fk X$ and the parameter describing the motion of a point on $\fk X$}. The strong residue theorem is crucial to the proof of this fact.

Let $\upphi:\scr W_{\fk X}(\Wbb_\blt)\rightarrow\scr O_{\mc B}$ be a conformal block associated to $\Wbb_\blt=\Wbb_1\otimes\cdots\otimes\Wbb_N$ and a family $\fk X=(\pi:\mc C\rightarrow\mc B;\sgm_1,\dots,\sgm_N)$ of $N$-pointed compact Riemann surfaces. Recall $\wr\mc C=\mc C\times_{\mc B}(\mc C\setminus\SX)$, $\wr\mc B=\mc C\setminus\SX$. The goal of this section is to prove the following theorem. 

\begin{thm}\label{lb4}
There is a unique $\scr O_{\mc C\setminus\SX}$-module morphism $\wr\upphi:\scr W_{\wr\fk X}(\Vbb\otimes\Wbb_\blt)\rightarrow\scr O_{\mc C\setminus\SX}$ satisfying the  following property: \index{zz@$\wr\upphi$}

"Choose any open subset $V\subset\mc B$  such that the restricted family $\fk X_V$ has local coordinates $\eta_1,\dots,\eta_N$ at $\sgm_1(V),\dots,\sgm_N(V)$. For each $j$, we choose a neighborhood $W_j\subset\mc C_V$ of $\sgm_j(V)$ on which $\eta_j$ is defined, such that $W_j$ intersects only $\sgm_j(V)$ among $\sgm_1(V),\dots,\sgm_N(V)$. Identify 
\begin{align*}
W_j=(\eta_j,\pi)(W_j)	\qquad\text{via }(\eta_j,\pi)
\end{align*}
so that $W_j$ is a neighborhood of $\{0\}\times V$ in $\Cbb\times V$. Let
\begin{align*}
U_j:=W_j\setminus\SX=W_j\setminus(\{0\}\times V)	
\end{align*}
which is inside $\Cbb^\times\times V$. Let $z$ be the standard coordinate of $\Cbb$. Identify
\begin{align*}
\scr W_{\fk X}(\Wbb_\blt)\big|_V=\Wbb_\blt\otimes_\Cbb\scr O_V	\qquad\text{via }\mc U(\eta_\blt).
\end{align*}
Identify 
\begin{align}
\scr W_{\wr\fk X}(\Vbb\otimes\Wbb_\blt)\big|_{U_j}=\Vbb\otimes\Wbb_\blt\otimes_\Cbb\scr O_{U_j}\qquad\text{via }\mc U(\triangle\eta_j,\wr\eta_\blt)\label{eq25}	
\end{align}
(cf. \eqref{eq17}). For each $u\in\Vbb,w_\blt\in\Wbb_\blt$, consider each vector of $\Wbb_\blt$ as a constant section of $\Wbb_\blt\otimes\scr O(U_j)$ and  $u\otimes w_\blt$ as a constant section of $\Vbb\otimes\Wbb_\blt\otimes_\Cbb\scr O(U_j)$. Then the following equation holds at the level of $\scr O(V)[[z^{\pm 1}]]$:
\begin{align}
\boxed{~\upphi\big(w_1\otimes\cdots\otimes Y(u,z)w_j\otimes\cdots\otimes w_N\big)=\wr\upphi(u\otimes w_\blt)~}\label{eq18}
\end{align}
where $Y(u,z)w:=\sum_{n\in\Zbb}Y(u)_nw\cdot z^{-n-1}$  is an element of $\Wbb_j((z))$, and  $\wr\upphi(u\otimes w_\blt)\in \scr O(U_j)$ is regarded as an element of $\scr O(V)[[z^{\pm 1}]]$ by taking Laurent series expansion."

Moreover, $\wr\upphi$ is a conformal block associated to $\wr\fk X$ and $\Vbb\otimes\Wbb_\blt$.
\end{thm}

Note that the left hand side of \eqref{eq18} is understood as
\begin{align*}
\sum_{n\in\Zbb} \upphi\big(w_1\otimes\cdots\otimes Y(u)_nw_j\otimes\cdots\otimes w_N\big)z^{-n-1},
\end{align*}
which is in $\scr O(U_j)((z))$. 

\begin{proof}[Proof of the uniqueness of $\wr\upphi$]
It suffices to restrict to the propagation of each fiber $\fk X_b$, i.e., restrict $\wr\upphi$ to a morphism $\upphi|_{\wr(\fk X_b)}:\scr W_{\wr(\fk X_b)}(\Vbb\otimes\Wbb_\blt)\rightarrow\scr O_{\mc C_b\setminus \SXb}$.	(Note that $\wr(\fk X_b)$ is $\mc C_b\times(\mc C_b\setminus\SXb)\rightarrow\mc C_b\setminus\SXb$ with marked points.) By \eqref{eq18}, we know $\wr\upphi|_{\wr(\fk X_b)}$ is uniquely determined on $(W_1\cup\cdots\cup W_N)\cap\mc C_b$. For two possible propagations $\wr_1\upphi,\wr_2\upphi$, let $\Omega$ be the set of all $x\in\mc C_b\setminus\SXb$ on a neighborhood of which $\wr_1\upphi|_{\wr(\fk X_b)}$ agrees with $\wr_2\upphi|_{\wr(\fk X_b)}$. Then $\Omega$ is open and  intersect any connected component of $\mc C_b$. By complex analysis, it is clear that if $U$ is a connected open subset of $\mc C_b\setminus\SXb$ intersecting $\Omega$ such that the restriction $\scr W_{\wr(\fk X_b)}(\Vbb\otimes\Wbb_\blt)|_U$ is equivalent to $\Vbb\otimes\Wbb_\blt\otimes_\Cbb\scr O_U$, then $U\subset\Omega$. So $\Omega$ is closed, and hence must be $\mc C_b\setminus\SXb$. This proves the uniqueness.
\end{proof}

\begin{proof}[Proof that \eqref{eq18} is independent of the choice of $\eta_\blt$]
Let us show that if \eqref{eq18} holds for all $u,w_\blt$ for a set of local coordinates $\eta_\blt$ defined on $W_1,\dots,W_N$, then it holds for another set $\eta_\blt'$. Indeed, it suffices to check this fact when restricted to each fiber $\fk X_b$. So we may assume that $\fk X$ is a single pointed Riemann surface $(C;x_1,\dots,x_N)$. Then \eqref{eq18} is equivalent to that for each $\nu\in H^0(W_j,\scr V_{\fk X}\otimes\omega_C(\star\SX))$, 
\begin{align*}
\upphi (w_1\otimes\cdots\otimes\nu\cdot w_j\otimes\cdots\otimes w_N)=\Res_{x_i}\wr\upphi(\nu\otimes w_\blt)
\end{align*}
where $\nu\cdot w_j$ is defined as in \eqref{eq63}. Then, as explained after \eqref{eq63}, this expression is independent of the choice of local coordinates.
\end{proof}

\begin{proof}[Proof  of the existence of $\wr\upphi$]
We identify $\scr W_{\wr\fk X}(\Vbb\otimes\Wbb_\blt)$ with $\scr V_{\fk X}\otimes\pi^*\scr W_{\fk X}(\Wbb_\blt)\big|_{\mc C\setminus\SX}$ as in Prop. \ref{lb5}, and construct  an $\scr O_{\mc C_\SX}$-module morphism $\wr\upphi:\scr V_{\fk X}\otimes\pi^*\scr W_{\fk X}(\Wbb_\blt)\big|_{\mc C\setminus\SX}\rightarrow\scr O_{\mc C\setminus\SX}$ satisfying \eqref{eq18}. By the uniqueness proved above, we can safely restrict the base manifold $\mc B$ to $V$. So we assume in the following that $\mc B=V$ and hence $\fk X$ has local coordinates $\eta_\blt$ at marked points. So we identify $\scr W_{\fk X(\Wbb_\blt)}$ with $\Wbb_\blt\otimes_\Cbb\scr O_{\mc B}$ through $\mc U(\eta_\blt)$, which yields
\begin{align}
\scr V_{\fk X}\otimes\pi^*\scr W_{\fk X}(\Wbb_\blt)=\scr V_{\fk X}\otimes_\Cbb\Wbb_\blt	\label{eq20}
\end{align}

For each $k\in\Nbb$, we let
\begin{align*}
	\scr E=(\scr V_{\fk X}^{\leq k})^\vee	
\end{align*}
be the dual bundle of $\scr V_{\fk X}^{\leq k}$. Then the identifications $W_j=(\eta_j,\pi)(W_j)$ and
\begin{align}
\scr V_{\fk X}^{\leq k}|_{W_j}=\Vbb^{\leq k}\otimes_\Cbb\scr O_{W_j} \qquad\text{via }\mc U_\varrho(\eta_j)	\label{eq24}
\end{align}
are compatible with the identifications in Sec. \ref{lb6} if we set the $E_i$ in that section to be $(\Vbb^{\leq k})^\vee$. Choose any $w_\blt\in\Wbb_\blt$. Let $s_j=\sum_{n\in\Zbb}e_{j,n}\cdot z^n$ as in Sec. \ref{lb6} where each $e_{j,n}\in (\Vbb^{\leq k})^\vee\otimes_\Cbb\scr O(\mc B)$ is defined by
\begin{align*}
u\in\Vbb^{\leq k}\mapsto\upphi\big(w_1\otimes\cdots\otimes Y(u)_{-n-1}w_j\otimes\cdots\otimes w_N\big)\in\scr O(\mc B).
\end{align*}
For each $b\in\mc B$, since $\upphi|_b$ is a conformal block, it vanishes on $H^0(\mc C_b,\scr V_{\mc C_b}^{\leq k}\otimes\omega_{\mc C_b}(\star\SXb))\cdot w_\blt$. This means that $s_1,\dots,s_N$ satisfy condition (c) of Theorem \ref{lb7}. Hence, by that theorem, $s_1,\dots,s_N$ are series expansions of a unique element $s\in H^0(\mc C,(\scr V_{\fk X}^{\leq k})^\vee(\star\SX))$, which restricts to $s\in H^0(\mc C\setminus\SX,(\scr V_{\fk X}^{\leq k})^\vee)$ and hence defines an $\scr O_{\mc C\setminus\SX}$-module morphism $\scr V_{\fk X}^{\leq k}|_{\mc C\setminus\SX}\otimes_\Cbb w_\blt\rightarrow\scr O_{\mc C\setminus\SX}$. These morphisms are compatible for different $k$, and is extended $\scr O_{\mc C\setminus\SX}$-linearly to a morphism $\wr\upphi:\scr V_{\fk X}\otimes\pi^*\scr W_{\fk X}(\Wbb_\blt)\big|_{\mc C\setminus\SX}\rightarrow \scr O_{\mc C\setminus\SX}$ (recall \eqref{eq20}). 

By Prop. \ref{lb5}, we can regard $\wr\upphi$ as a morphism $\wr\upphi:\scr W_{\wr\fk X}(\Vbb\otimes\Wbb_\blt)\rightarrow\scr O_{\mc C\setminus\SX}$. Note that the identifications \eqref{eq20} and \eqref{eq24} are compatible with \eqref{eq25}, thanks to the commutative diagram \eqref{eq26}. Thus, $\wr\upphi$  satisfies \eqref{eq18} under the required identifications with respect to the local coordinates $\eta_\blt$. By the previous step, $\wr\upphi$  satisfies \eqref{eq18} for any other choice of local coordinates.
\end{proof}

\begin{proof}[Proof that $\wr\upphi$ is a conformal block]
Since being a conformal block is a fiberwise condition, we may prove $\wr\upphi$ is a conformal block by restricting it to each fiber $\fk X_b$ and its propagation $\wr(\fk X_b)$. Therefore, we may assume that $\mc B$ is a single point. So $C:=\mc C$ is a compact Riemann surface. We trim each $W_j$ so that $\eta_j(W_j)=\mc D_{r_j}$ for some $r_j>0$.

From the previous proof, we have a morphism $\wr\upphi:\scr W_{\wr\fk X}(\Vbb\otimes\Wbb_\blt)\rightarrow\scr O_{C\setminus\SX}$ which, given the trivializations in the statement of Theorem \ref{lb4}, is equal to \eqref{eq18} when restricted to $W_j\setminus\SX=W_j\setminus\{\sgm_j\}$. This shows that the series \eqref{eq18} converges a.l.u. on $0<|z|<r_j$. Therefore, as explained in Example \ref{lb9}, we can use Thm. \ref{lb8} to conclude that $\wr\upphi$ is a conformal block when restricted to each $W_j$. By Prop. \ref{lb10}, $\wr\upphi$ is globally a conformal block.
\end{proof}

The proof of Thm. \ref{lb4} is completed.

We now give an application of this theorem. Suppose $\Ebb$ is a set of vectors in a $\Vbb$-module $\Wbb$. We say \textbf{$\Ebb$ generates $\Wbb$} if $\Wbb$ is spanned by vectors of the form $Y(u_1)_{n_1}\cdots Y(u_k)_{n_k}w$ where $k\in\Zbb_+$, $u_1,\dots,u_k\in\Vbb$, $n_1,\dots,n_k\in\Zbb$, $w\in\Ebb$.

\begin{pp}\label{lb12}
Let $\fk X=(C;x_1,\dots,x_N)$ be an $N$-pointed connected compact Riemann surface, where $N\geq 2$. Choose local coordinate $\eta_j$ at $x_j$. Associate $\Vbb$-modules $\Wbb_1,\dots,\Wbb_N$ to $x_1,\dots,x_N$. Identify $\scr W_{\fk X}(\Wbb_\blt)=\Wbb_1\otimes\cdots\otimes\Wbb_N$ via $\mc U(\eta_\blt)$. Suppose that for each $2\leq i\leq N$, $\Ebb_i$ is a generating subset of $\Wbb_i$. Then any conformal block $\upphi:\Wbb_1\otimes\Wbb_2\otimes\cdots\Wbb_N\rightarrow\Cbb$ is determined by its values on $\Wbb_1\otimes\Ebb_2\otimes\cdots\otimes\Ebb_N$.
\end{pp}

\begin{proof}
Assume $\upphi$ vanishes on $\Wbb_1\otimes\Ebb_2\otimes\cdots\otimes\Ebb_N$. We shall show that $\upphi$ vanishes on $\Wbb_1\otimes Y(u)_n\Ebb_2\otimes\cdots\otimes\Ebb_N$ for each $u\in\Vbb,n\in\Zbb$. Then, by successively applying this result, we see that $\upphi$ vanishes on $\Wbb_1\otimes\Wbb_2\otimes\Ebb_3\otimes\cdots\otimes\Ebb_N$, and hence (by repeating again this procedure several times) vanishes on $\Wbb_1\otimes\Wbb_2\otimes\cdots\otimes\Wbb_N$.

Identify $\scr W_{\wr\fk X}(\Vbb\otimes\Wbb_\blt)=\scr V_{\fk X}\big|_{C\setminus\SX}\otimes_\Cbb\Wbb_\blt$ using \eqref{eq32}. Then we can consider $\wr\upphi$ as a morphism $\wr\upphi:\scr V_{\fk X}\big|_{C\setminus\SX}\otimes_\Cbb\Wbb_\blt\rightarrow\scr O_{C\setminus\SX}$. Let $\Omega$ be the open set of all $x\in C\setminus\SX$ such that $x$ has a neighborhood $U\subset C\setminus\SX$ such that the restriction
\begin{align*}
\wr\upphi|_U:\scr V_{\fk X}|_U\otimes_\Cbb\Wbb_1\otimes\Ebb_2\otimes\cdots\otimes\Ebb_N\rightarrow\scr O_U	
\end{align*}
vanishes. We note that if $U$ is connected, and if we can find an injective $\eta\in\scr O(U)$ (so that $\scr V_{\fk X}|_U$ is trivialized to $\Vbb\otimes_\Cbb\scr O_U$), then by complex analysis, $\wr\upphi|_U$ vanishes whenever $\wr\upphi|_V$ vanishes for some non-empty open $V\subset U$. We conclude that if such $U$ intersects $\Omega$, then $U$ must be inside $\Omega$. So $\Omega$ is closed. It is clear that for each $w_1\in\Wbb_1,w_2\in\Ebb_2,\dots,w_N\in\Ebb_N$, the following formal series of $z$
\begin{align*}
\upphi(Y(u,z)w_1\otimes w_2\otimes\cdots\otimes w_N)	
\end{align*}
vanishes. Thus, by Thm. \ref{lb4}, $\Omega$ contains $W_0\setminus\{x_0\}$ for some neighborhood $W_0$ of $x_0$. Therefore $\Omega=C\setminus\SX$. By Thm. \ref{lb4} again, we see
\begin{align*}
\upphi(w_1\otimes Y(u,z)w_2\otimes\cdots\otimes w_N)	
\end{align*}
also vanishes. This finishes the proof.
\end{proof}

\begin{rem}\label{lb14}
Since $\id$ generates $\Vbb$, we see that if $\Vbb,\Wbb_2,\dots,\Wbb_N$ (where $N\geq 2$) are associated to a connected $\fk X=(C;x_1,\dots,x_N)$, then any conformal block $\upphi:\Vbb\otimes\Wbb_2\otimes\cdots\otimes\Wbb_N\rightarrow\Cbb$ is determined by its values on $\id\otimes\Wbb_2\otimes\cdots\otimes\Wbb_N$. This proves the following two well-known results. In fact, in the literature, the propagation of conformal blocks is best known in the form of the following two corollaries.
\end{rem}

\begin{co}\label{lb13}
Let $\fk X=(C;x_1,\dots,x_N)$ be an $N$-pointed compact Riemann surface associated with $\Vbb$-module $\Wbb_1,\dots,\Wbb_N$. Identify $\scr W_{\wr\fk X}(\Vbb\otimes\Wbb_\blt)=\scr V_{\fk X}\big|_{\mc C\setminus\SX}\otimes_\Cbb\scr W_{\fk X}(\Wbb_\blt)$ via \eqref{eq32}. Then for each $x\in\mc C\setminus\SX$, $\wr\upphi|_x$ is the unique linear map $\scr V_{\fk X}\big|_x\otimes_\Cbb\scr W_{\fk X}(\Wbb_\blt)\rightarrow\Cbb$ which is a conformal block and satisfies
\begin{align*}
\wr\upphi|_x(\id\otimes w)=\upphi(w)	
\end{align*}
for each vector $w\in\scr W_{\fk X}(\Wbb_\blt)$.
\end{co}

\begin{proof}
The uniqueness follows from the previous remark. We shall show that $\wr\upphi(\id\otimes w)$, which is an element of $\scr O(C\setminus\SX)$, equals the constant function $\upphi(w)$. By complex analysis, it suffices to prove $\wr\upphi(\id\otimes w)=\upphi(w)$ when restricted to each $W_j\setminus\{x_j\}$ (where $W_j$ is a small disc containing $x_j$ on which a local coordinate is defined). This is true by \eqref{eq18}.
\end{proof}

\begin{co}\label{lb25}
Let $\fk X=(C;x_1,\dots,x_N)$ be an $N$-pointed connected compact Riemann surface associated with $\Vbb$-module $\Wbb_1,\dots,\Wbb_N$. Choose $x\in C\setminus\{x_1,\dots,x_N\}$. Then the space of conformal blocks associated to $\fk X$ and $\Wbb_\blt$ is isomorphic to the space of conformal blocks associated to $(\wr\fk X)_x=(C;x,x_1,\dots,x_N)$ and $\Vbb,\Wbb_1,\dots,\Wbb_N$.
\end{co}

\begin{proof}
We assume the identifications in Cor. \ref{lb13}. The linear map $F$ from the first space to the second one is defined by $\upphi\mapsto\wr\upphi|_x$. The linear map $G$ from the second one to the first one is defined by $\uppsi\mapsto\uppsi(\id\otimes\cdot)$. By Cor. \ref{lb13}, we have $G\circ F=1$. By Remark \ref{lb14}, $G$ is injective. So $G$ is bijective.
\end{proof}

\section{Multi-propagation}\label{lb20}

Let $\fk X=(C;x_1,\dots,x_N)$ be an $N$-pointed compact Riemann surface. Recall $\SX=\{x_1,\dots,x_N\}$. We choose local coordinates $\eta_1\in\scr O(W_1),\dots,\eta_N\in\scr O(W_N)$ of $\fk X$ at $x_1,\dots,x_N$, where each $W_j$ is a neighborhood of $x_j$ satisfying $W_j\cap\SX=\{x_j\}$.

Let $n\in\Zbb_+$. By Section \ref{lb2}, $\wr^n\fk X$ is 
\begin{align*}
\wr^n\fk X=(\wr^n\pi:C\times\Conf^n(C\setminus\SX)\rightarrow\Conf^n(C\setminus\SX);\sigma_1,\dots,\sigma_n,\wr^nx_1,\dots,\wr^nx_N) 
\end{align*}
where $\wr^n\pi$ is the projection onto the second component, and the sections are given by
\begin{gather*}
\wr^nx_j(y_1,\dots,y_n)=(x_j,y_1,\dots,y_n),	\\
\sigma_i(y_1,\dots,y_n)=(y_i,y_1,\dots,y_n).	
\end{gather*}
We define local coordinate \index{zz@$\wr\eta_j$, $\wr^n\eta_j$}
\begin{align}
	\wr^n\eta_j(x,y_1,\dots,y_n)=\eta_j(x)
\end{align}
of $\wr^n\fk X$ at $x_j\times\Conf^n(C\setminus\SX)$, defined on $W_j\times\Conf^n(C\setminus\SX)$. Suppose $U$ is an open subset of $C\setminus\SX$ which admits an injective $\mu\in\scr O(U)$. Then a local coordinate $\triangle_i\mu$ of $(\wr^n\fk X)_U$ at $\sigma_i(U)$ is defined by \index{zz@$\triangle\mu$, $\triangle_i\mu$}
\begin{align}
\triangle_i\mu(x,y_1,\dots,y_n)=\mu(x)-\mu(y_i)	
\end{align}
whenever this expression is definable.

We shall relate the $\scr W$-sheaves with the exterior product $\scr V_C^{\boxtimes n}$, which is an $\scr O_{C^n}$-module defined by
\begin{align}
\scr V_C^{\boxtimes n}:=	\pr_1^*\scr V_C\otimes\pr_2^*\scr V_C\otimes\cdots\otimes\pr_n^*\scr V_C.
\end{align}
Here, each $\pr_i:C^n=\underbrace{C\times\cdots\times C}_n\rightarrow C$ is the projection onto the $i$-th component. The tensor products are over $\scr O_{C^n}$ as usual. Similar to the description in Section \ref{lb15}, the $\scr O_{C^n}$-module $\pr_i^*\scr V_C$ is the pullback of the (infinite-rank) vector bundle $\scr V_C$ along $\pr_i$ to $C^n$, i.e., $\scr V_C\otimes_{\scr O_{C}}\scr O_{C^n}$ where the action of $f\in\scr O_C$ on $\scr O_{C^n}$ is defined by the multiplication of $f\circ\pr_i$. If $U\subset C$ is open and $\mu\in\scr O(U)$ is injective, we then have a trivilization
\begin{align*}
\pr_i^*\mc U_\varrho(\mu):\pr_i^*\scr V_C\big|_{\pr_i^{-1}(U)}\xrightarrow{\simeq}\Vbb\otimes_\Cbb\scr O_{\pr_i^{-1}(U)}.	
\end{align*}

\begin{pp}\label{lb16}
We have a unique isomorphism
\begin{align}
\scr W_{\wr^n\fk X}(\Vbb^{\otimes n}\otimes\Wbb_\blt)\xrightarrow{\simeq}\scr V_C^{\boxtimes n}\big|_{\Conf^n(C\setminus\SX)}\otimes_\Cbb \scr W_{\fk X}(\Wbb_\blt)	
\end{align}
such that for any $n$ mutually disjoint open subsets $U_1,\dots,U_n\subset C\setminus\SX$ and any injective $\mu_1\in\scr O(U_1),\dots,\mu_n\in\scr O(U_n)$, the restriction of this isomorphism to $U$ makes the following diagram commutes.
\begin{equation}
\begin{tikzcd}[column sep=tiny]
\scr W_{\wr^n\fk X}(\Vbb^{\otimes n}\otimes\Wbb_\blt)\big|_{U_1\times\cdots\times U_n} \arrow[rr,"\simeq"] \arrow[ddr, "\simeq" ,"{\mc U(\triangle_\blt\mu_\blt,\wr^n\eta_\blt)}"']  && \scr V_C^{\boxtimes n}\big|_{U_1\times\cdots\times U_n}\otimes_\Cbb\scr W_{\fk X}(\Wbb_\blt)	 \arrow[ddl,"\simeq"',"{\pr_1^*\mc U_\varrho(\mu_1)\otimes\cdots\otimes\pr_n^*\mc U_\varrho(\mu_n)\otimes \mc U(\eta_\blt)}"]\\
{}\\
& \Vbb^{\otimes n}\otimes\Wbb_\blt\otimes_\Cbb \scr O_{U_1\times\cdots\times U_n}
\end{tikzcd}\label{eq33}	
\end{equation}
Here,
\begin{gather*}
(\triangle_\blt\mu_\blt,\wr^n\eta_\blt):=(\triangle_1\mu_1,\dots,\triangle_n\mu_n,\wr^n\eta_1,\dots,\wr^n\eta_n).	
\end{gather*}
Moreover, the isomorphism is independent of the choice of $\eta_\blt$.
\end{pp}

\begin{proof}
Suppose we have another injective $\mu_i'\in\scr O(U_i)$. Similar to the proof of Lemma \ref{lb3}, we see that for each $y_i\in U_i$,
\begin{align*}
(\triangle_i\mu_i|\triangle_i\mu_i')_{(y_1,\dots,y_n)}=\varrho(\mu_i|\mu_i')_{y_i}.	
\end{align*}
(See  \eqref{eq10} and \eqref{eq8} for the meaning of notations.) Using this relation, one shows, as in the proof of Prop. \ref{lb5}, that the transition functions for the two trivializations in \eqref{eq33} are equal. This finishes the proof.
\end{proof}

Choose a conformal block $\upphi:\Wbb_\blt\rightarrow\Cbb$ associated to $\fk X$ and $\scr W_{\fk X}(\Wbb_\blt)$. By Theorem \ref{lb4}, we have $n$-propagation $\wr^n\upphi$ defined inductively by
\begin{align*}
\wr^n\upphi=\wr(\wr^{n-1}\upphi)
\end{align*}
which is a conformal block associated to $\wr^n\fk X$ and $\Vbb^{\otimes n}\otimes\Wbb_\blt$. By Prop. \ref{lb16}, we can regard $\wr^n\upphi$ as a morphism
\begin{align*}
\wr^n\upphi:\scr V_C^{\boxtimes n}\big|_{\Conf^n(C\setminus\SX)}\otimes_\Cbb\scr W_{\fk X}(\Wbb_\blt)	\rightarrow\scr O_{\Conf^n(C\setminus\SX)}.	
\end{align*}

\subsection*{Important facts about $\wr^n\upphi$}

Choose (non-necessarily disjoint) open $U_1,\dots,U_n\subset C$ and write \index{Conf@$\Conf(U_\blt\setminus\SX)$}
\begin{align*}
\Conf(U_\blt\setminus\SX)=(U_1\times\cdots\times U_n)\cap\Conf^n(C\setminus\SX).	
\end{align*}
For any sections $v_i\in\scr V_C(U_i)$ and any $w\in\scr W_{\fk X}(\Wbb_\blt)$, we write \index{zz@$\wr^n\upphi$}
\begin{align}
\wr^n\upphi(v_1,\dots,v_n,w):=\wr^n\upphi\Big(\pr_1^*v_1\otimes\cdots\otimes\pr_n^*v_n\otimes w\big|_{\Conf(U_\blt\setminus\SX)}\Big)&\nonumber\\[1ex]
\in\scr O\big(\Conf(U_\blt\setminus\SX)\big).&
\end{align}
We now summarize some important properties of $\wr^n\upphi$ in this setting.

As an elementary fact, the map $(v_1,\dots,v_n)\mapsto\wr^n\upphi(v_1,\dots,v_n,w)$ intertwines the action of each $\scr O(U_i)$ on the $i$-th component. (Here, each $f\in\scr O(U_i)$ acts on $\scr O(\Conf(U_\blt\setminus\SX))$ by the multiplication of $(f\circ\pr_i)|_{\Conf(U_\blt\setminus\SX)}$). Moreover, it is compatible with restricting to open subsets of $U_i$.

We set $\wr^0\upphi=\upphi$.

\begin{thm}\label{lb17}
Identify
\begin{align*}
\scr W_{\fk X}(\Wbb_\blt)=\Wbb_\blt	\qquad\text{via }\mc U(\eta_\blt).
\end{align*}
Choose any $w_\blt\in\Wbb_\blt$. For each $1\leq i\leq n$, choose an open subset  $U_i$  of $C$ equipped with an injective $\mu_i\in\scr O(U_i)$. Identify 
\begin{align*}
\scr V_C\big|_{U_i}=\Vbb\otimes_\Cbb\scr O_{U_i}\qquad\text{via }\mc U_\varrho(\mu_i).
\end{align*}
Choose $v_i\in\scr V_C(U_i)=\Vbb\otimes_\Cbb\scr O(U_i)$. Choose $(y_1,\dots,y_n)\in\Conf(U_\blt\setminus\SX)$. Then the following are true.
\begin{enumerate}[label=(\arabic*)]
\item If $U_1=W_j$ (where $1\leq j\leq N$) and contains only $y_1,x_j$ among all $x_\blt,y_\blt$, if $\mu_1=\eta_j$, and if $U_1$ contains the closed disc with center $x_j$ and radius $|\eta_j(y_1)|$ (under the coordinate $\eta_j$), then
\begin{align}
&\wr^n\upphi(v_1,v_2,\dots,v_n,w_\blt)\big|_{y_1,y_2,\dots,y_n}\nonumber\\
=&\wr^{n-1}\upphi\big(v_2,\dots,v_n,w_1\otimes\cdots\otimes Y(v_1,z)w_j\otimes\cdots\otimes w_N\big)\big|_{y_2,\dots,y_n}~\big|_{z=\eta_j(y_1)}\label{eq35}
\end{align}
where the series of $z$ on the right hand side converges absolutely, and $v_1$ is considered as an element of $\Vbb\otimes\Cbb((z))$ by taking Taylor series expansion with respect to the variable $\eta_j$ at $x_j$.
\item If $U_1=U_2$ and contains only $y_1,y_2$ among all $x_\blt,y_\blt$, if $\mu_1=\mu_2$, and if $U_2$ contains the closed disc with center $y_2$ and radius $|\mu_2(y_1)-\mu_2(y_2)|$ (under the coordinate $\mu_2$), then
\begin{align}
&\wr^n\upphi(v_1,v_2,v_3,\dots,v_n,w_\blt)\big|_{y_1,y_2,\dots,y_n}\nonumber\\
=&\wr^{n-1}\upphi\big(Y(v_1,z)v_2,v_3,\dots,v_n,w_\blt\big)\big |_{y_2,\dots,y_n}~\big |_{z=\mu_2(y_1)-\mu_2(y_2)}\label{eq36}
\end{align}
where the series of $z$ on the right hand side converges absolutely, and $v_1$ is considered as an element of $\Vbb\otimes\Cbb((z))$ by taking Taylor series expansion with respect to the variable $\mu_2-\mu_2(y_2)$ at $y_2$.
\item  We have
\begin{align}
\wr^n\upphi(\id,v_2,v_3,\dots,v_n,w_\blt)=\wr^{n-1}\upphi(v_2,\dots,v_n,w_\blt).
\end{align}
\item For any permutation $\uppi$ of the set $\{1,2,\dots,n\}$, we have
\begin{align}
\wr^n\upphi(v_{\uppi(1)},\dots,v_{\uppi(n)},w_\blt)\big|_{y_{\uppi(1)},\dots,y_{\uppi(n)}}=\wr^n\upphi(v_1,\dots,v_n,w_\blt)\big|_{y_1,\dots,y_n}.
\end{align}
\end{enumerate}
\end{thm}


\begin{proof}
When $v_1,v_2$ are constant sections (i.e. in $\Vbb$), (1) and (2) follow from Thm. \ref{lb4} and especially formula \eqref{eq18}. The general case follows immediately. (3) follows from Cor. \ref{lb13}. By  (3), part (4) holds when $v_1,\dots,v_n$ are all the vacuum section $\id$. Thus, it hols for all $v_1,\dots,v_n$ due to Prop. \ref{lb12}.
\end{proof}

\section{Sewing and multi-propagation}

We assume, in addition to the setting of Section \ref{lb18}, that $\wtd{\mc B}$ is a single point. Namely, we have an $(N+2M)$-pointed compact Riemann surface
\begin{align*}
\wtd{\fk X}=(\wtd C;x_1,\dots,x_N;x'_1,\dots,x_M';x_1'',\dots,x_M''),	
\end{align*}
where each connected component of $\wtd C$ contains one of $x_1,\dots,x_N$. For each $1\leq j\leq M$, $\wtd {\fk X}$ has  local coordinates $\xi_j$ at $x_j'$ and $\varpi_j$ at $x_j''$ defined respectively on neighborhoods $W_j'\ni x_j',W_j''\ni x_j''$. All $W_j',W_j''$ (where $1\leq j\leq M$) are mutually disjoint and do not contain $x_1,\dots,x_N$. $\xi_j(W_j')=\mc D_{r_j}$, and $\varpi_j(W_j'')=\mc D_{\rho_j}$. For each marked point $x_i$ we associate a $\Vbb$-module $\Wbb_i$. To $x_j'$ and $x_j''$ to we associate respectively a $\Vbb$-module $\Mbb_j$ and its contragredient $\Mbb_j'$. We set
\begin{align*}
S_{\wtd{\fk X}}=\{x_1,\dots,x_N\}.	
\end{align*}
Also, for each $1\leq i\leq N$, choose a local coordinate $\eta_i$ at $x_i$. Identify
\begin{align*}
\scr W_{\wtd{\fk X}}(\Wbb_\blt\otimes\Mbb_\blt\otimes\Mbb'_\blt)=\Wbb_\blt\otimes\Mbb_\blt\otimes\Mbb'_\blt	\qquad\text{via }\mc U(\eta_\blt,\xi_\blt,\varpi_\blt).
\end{align*}

We sew $\wtd{\fk X}$ along each $x_j',x_j''$ to obtain a family
\begin{align*}
\fk X=(\pi:\mc C\rightarrow\mc D_{r_\blt\rho_\blt}^\times;x_1,\dots,x_N),	
\end{align*}
where the points $x_1,\dots,x_N$  on $\wtd C$  and the local coordinates $\eta_1,\dots,\eta_N$ at these points extend constantly (over $\mc D_{r_\blt\rho_\blt}^\times$) to sections and local coordinates of $\fk X$, denoted by the same symbols. (Cf. Sec. \ref{lb18}.) For each $q_\blt\in\mc D_{r_\blt\rho_\blt}^\times$, we identify
\begin{align*}
\scr W_{\fk X_{q_\blt}}(\Wbb_\blt)=\Wbb_\blt\qquad\text{via }\mc U(\eta_\blt).
\end{align*}

Let $\upphi:\Wbb_\blt\otimes\Mbb_\blt\otimes\Mbb_\blt'\rightarrow\Cbb$ be a conformal block associated to $\wtd{\fk X}$ that converges a.l.u. on $\mc D_{r_\blt\rho_\blt}^\times$. Let $U_1,\dots,U_n\subset \wtd C$ be open and disjoint from each $W_j',W_j''$. For each $q_\blt\in\mc D_{r_\blt\rho_\blt}^\times$, since the fiber $\mc C_{q_\blt}$ is obtained by removing a small part of each $W_j',W_j''\subset\wtd C$ and gluing the remaining part of $\wtd C$, we see that each $U_i$ can be regarded as an open subset of the fiber $\mc C_{q_\blt}$. By Thm. \ref{lb8},  \index{S@$\wtd{\mc S}_{q_\blt}\upphi$} 
\begin{align*}
\wtd{\mc S}_{q_\blt}\upphi:=\wtd{\mc S}\upphi|_{q_\blt}	
\end{align*}
is a conformal block associated to $\fk X_{q_\blt}$. Thus, we can consider its $n$-propagation $\wr^n\wtd{\mc S}_{q_\blt}\upphi$.  In the setting of Thm. \ref{lb17}, and setting
\begin{align*}
\Conf(U_\blt\setminus\SXtd)=(U_1\times\cdots\times U_n)\cap\Conf^n(\wtd C\setminus\SXtd),	
\end{align*}
for each $v_i\in\scr V_{\wtd C}(U_i)=\scr V_{\mc C_{q_\blt}}(U_i)$ and $w_\blt\in\Wbb_\blt$,
\begin{align*}
\wr^n\wtd{\mc S}_{q_\blt}\upphi(v_1,\dots,v_n,w_\blt)\qquad\in\scr O(\Conf(U_\blt\setminus\SXtd)).
\end{align*}
This expression relies holomorphically on $q_\blt$ due to  Thm. \ref{lb4} (applied $n$ times). Thus, by varying $q_\blt$, we obtain
\begin{align}
\wr^n\wtd{\mc S}\upphi(v_1,\dots,v_n,w_\blt)\qquad\in\scr O\big(\mc D_{r_\blt\rho_\blt}^\times\times\Conf(U_\blt\setminus\SXtd)\big).	\label{eq37}
\end{align}

Since $\wr^n\upphi$ is a conformal block associated to $\wr^n\wtd{\fk X}$, we can talk about the a.l.u. convergence of its sewing $\wtd{\mc S}\wr^n\upphi$, which is a conformal block by Thm. \ref{lb8} again. In the setting of Thm. \ref{lb17}, this means  for each $v_i\in\scr V_{\wtd C}(U_i)$ and $w_\blt\in\Wbb_\blt$ the a.l.u. convergence of
\begin{align}
\wtd{\mc S}\wr^n\upphi(v_1,\dots,v_n,w_\blt)&:=\wr^n\upphi\Big(v_1,\dots,v_n,w_\blt\otimes(q_1^{\wtd  L_0}\btr\otimes_1\btl)\otimes\cdots\otimes (q_M^{\wtd  L_0}\btr\otimes_M\btl)\Big)	\label{eq39}\\
&\in\scr O(\Conf(U_\blt\setminus\SXtd))[[q_1,\dots,q_M]]\nonumber
\end{align}
on $\mc D_{r_\blt\rho_\blt}^\times\times \Conf(U_\blt\setminus\SXtd)$ in the sense of Def. \ref{lb19}. We may ask whether this convergence is true, and if it is true, whether the value of this expression at $q_\blt$ equals \eqref{eq37}. The answer is Yes.

\begin{thm}\label{lb21}
Assume $\wtd{\mc S}\upphi$ converges a.l.u. on $\mc D_{r_\blt\rho_\blt}^\times$. Then for each open $U_1,\dots,U_n\subset \wtd C$ disjoint from $W_j',W_j''$ ($1\leq j\leq N$), each $v_i\in\scr V_{\wtd C}(U_i)$ and $w_\blt\in\Wbb_\blt$, the relation
\begin{align}
\wtd{\mc S}\wr^n\upphi(v_1,\dots,v_n,w_\blt)=	\wr^n\wtd{\mc S}\upphi(v_1,\dots,v_n,w_\blt)\label{eq38}
\end{align}
holds at the level of $\in\scr O(\Conf(U_\blt\setminus\SXtd))[[q_1^{\pm 1},\dots,q_M^{\pm 1}]]$. In particular, the left hand side converges a.l.u. on $\mc D_{r_\blt\rho_\blt}^\times\times\Conf(U_\blt\setminus\SXtd)$.
\end{thm}
We note that the right hand side of \eqref{eq38} is considered as a series of $q_1,\dots,q_M$ by taking Laurent series expansion.

\begin{proof}
We prove this theorem by induction on $n$. Let us assume the case for  $n-1$ is proved. For each $1\leq i\leq N$ we choose a neighborhood $W_i\subset\wtd C$ of $x_i$ on which $\eta_i$ is defined. We assume $W_i$ is small enough such that it does not intersect any $W_j',W_j''$ ($1\leq j\leq N$) and contains only $x_1$ of $x_1,\dots x_N$.

Step 1.  Note that we can clearly shrink $\mc D_{r_\blt\rho_\blt}^\times$ since the formal series in \eqref{eq38} are independent of the size of this punctured polydisc. Therefore, we can also shrink each $W_j',W_j''$ to smaller discs, so that the interior of $\wtd C\setminus\bigcup_{1\leq j\leq M}(W_j'\cup W_j'')$ (denoted by $\mbf H$) is homotopic to $\mbf H_0=\wtd C\setminus\{x_1',\dots,x_M',x_1'',\dots,x_M''\}$. Therefore, since each connected component of $\wtd C$ (and hence each one of $\mbf H_0$) intersects $x_1,\dots,x_N$, each one of $\mbf H_0$ contains at least one of $W_1,\dots,W_N$. The same is true for $\mbf H$. So each connected component of $\mbf H\setminus\SXtd$ contains at lease one $W_j\setminus\{x_j\}$.

Fix $U_2,\dots,U_n$ and $v_2,\dots,v_n,w_\blt$ as in the statement of this theorem. Let $\Omega$ be the open set of all $y_1\in\mbf H\setminus\SXtd$ contained in an open $U_1\subset\mbf H\setminus\SXtd$ such that  \eqref{eq38} holds for all $v_1\in\scr V_{\wtd C}(U_1)$. By complex analysis, if $V_1\subset\mbf H\setminus\SXtd$ is open such that $\scr V_{\wtd C}|_{V_1}$ is trivializable (e.g., when there is an injective element of $\scr O(V_1)$), then $V_1\subset\Omega$ whenever $V_1\cap \Omega\neq\emptyset$. So $\Omega$ is closed. Thus, if $\Omega$ intersects $W_1\setminus\{x_1\},\dots,W_N\setminus\{x_N\}$, then $\Omega=\mbf H\setminus\SXtd$, which finishes the proof.

Step 2. We show $\Omega$ intersects $W_1\setminus\{x_1\}$, and hence intersects the other $W_i\setminus\{x_i\}$ by a similar argument. Indeed, we shall show that \eqref{eq38} holds whenever $U_1=W_1$.

Note $w_\blt=w_1\otimes w_2\otimes\cdots\otimes w_N$ by convention. We let $w_\circ=w_2\otimes\cdots \otimes w_N$. Identify $W_1$ with $\eta_1(W_1)$ via $\eta_1$ so that $\eta_1$ is identified with the standard coordinate $z$. Let $\Conf(U_\circ\setminus\SXtd)=(U_2\times\cdots\times U_n)\cap\Conf^{n-1}(\wtd C\setminus\SXtd)$. Identify $\scr V_{\wtd C}|_{W_1}$ with $\Vbb\otimes_\Cbb\scr O_{W_1}$ using $\mc U_\varrho(\eta_1)$. Choose any $v_1\in\Vbb\otimes_\Cbb\scr O(W_1)$. Then by Thm. \ref{lb17},
\begin{align*}
\wtd{\mc S}\wr^n\upphi(v_1,v_2,\dots,v_n,w_\blt)=\wtd{\mc S}\wr^{n-1}\upphi(v_2,\dots,v_n,Y(v_1,z)w_1\otimes w_\circ)	
\end{align*}
at the level of $\scr O(\Conf(U_\circ\setminus\SXtd))[[z^{\pm 1},q_1^{\pm1},\dots,q_M^{\pm1}]]$. By our assumption on the $(n-1)$-case, this expression can be regarded as an element of (and hence this equation holds at the level of) $\scr O(\mc D_{r_\blt\rho_\blt}^\times\times\Conf(U_\circ\setminus\SXtd))[[z^{\pm 1}]]$, and we have
\begin{align*}
\wtd{\mc S}\wr^n\upphi(v_1,v_2,\dots,v_n,w_\blt)=\wr^{n-1}\wtd{\mc S}\upphi(v_2,\dots,v_n,Y(v_1,z)w_1\otimes w_\circ)	
\end{align*}
also on this level. By Thm. \ref{lb17} again, this expression equals
\begin{align*}
\wr^n\wtd{\mc S}\upphi(v_1,v_2,\dots,v_n,w_1\otimes w_\circ)	
\end{align*}
on this level. Since the above is an element of $\scr O(\mc D_{r_\blt\rho_\blt}^\times\times\Conf(U_\blt\setminus\SXtd))$, by the uniqueness of Laurent series expansion, we see the left hand side of \eqref{eq38} is also an element of this ring, and \eqref{eq38} holds on this level.
\end{proof}

\begin{rem}\label{lb23}
We discuss how to generalize Thm. \ref{lb21} to the case that $\wtd{\fk X}$ is a family of compact Riemann surfaces as in Sec. \ref{lb18}. We assume the setting of that section, together with one more assumption that $\wtd{\fk X}$ has local coordinates $\eta_1,\dots,\eta_N$ at $\sgm_1(\wtd{\mc B}),\dots,\sgm_N(\wtd{\mc B})$ so that we can identify the $\scr W$-sheaves with the free ones using the trivialization $\mc U(\eta_\blt)$ or $\mc U(\eta_\blt,\xi_\blt,\varpi_\blt)$.

We use freely the notations in Sec. \ref{lb18}. Let $\SXtd=\bigcup_{1\leq i\leq M}\sgm_i(\wtd{\mc B})$. Let
\begin{align*}
\upphi:\Wbb_\blt\otimes\Mbb_\blt\otimes\Mbb_\blt'\otimes_\Cbb\scr O_{\wtd{\mc B}}\rightarrow \scr O_{\wtd{\mc B}}	
\end{align*}
be a conformal block associated to $\wtd{\fk X}$ converging a.l.u. on $\mc B=\mc D_{r_\blt\rho_\blt}^\times\times\wtd{\mc B}$. Choose any open $U_1,\dots,U_n\subset\wtd{\mc C}$ disjoint from all $W_j',W_j''$. Choose $v_i\in\scr V_{\wtd{\fk X}}(U_i)$ and $w_\blt\in\Wbb_\blt$. Let $\Conf_{\wtd{\mc B}}(U_\blt\setminus\SXtd)$ be the set of all $(y_1,\dots,y_n)\in\Conf(U_\blt\setminus\SXtd)$ satisfying $\wtd\pi(y_1)=\cdots=\wtd\pi(y_n)$.  For each $m_j\in\Mbb_j,m_j'\in\Mbb_j'$, we have
\begin{align*}
	\wr^n\upphi(v_1,\dots,v_n,w_\blt\otimes m_\blt\otimes m_\blt')\qquad \in\scr O(\Conf_{\wtd{\mc B}}(U_\blt\setminus\SXtd))	
\end{align*}
whose restriction to each $\wtd{\mc C}_b^{\times n}$ (where $b\in\wtd{\mc B}$ is such that $\wtd{\mc C_b}$ intersects $U_1,\dots,U_n$) is $\wr^n(\upphi|_b)(v_1,\dots,v_n,w_\blt\otimes m_\blt\otimes m_\blt')$. (Indeed, this expression is a priori only a function holomorphic when restricted to each $\wtd{\mc C}_b^{\times n}$; that it is holomorphic on $\Conf_{\wtd{\mc B}}(U_\blt\setminus\SXtd)$ (i.e., holomorphic when $b$ also varies) is due to Thm. \ref{lb4}.) Thus, we can define
\begin{align}
\wtd{\mc S}\wr^n\upphi(v_1,\dots,v_n,w_\blt)\qquad\in\scr O(\Conf_{\wtd{\mc B}}(U_\blt\setminus\SXtd))[[q_1^{\pm 1},\dots,q_M^{\pm 1}]] \label{eq40}
\end{align}
using \eqref{eq39}. Similarly, with the aid of Thm. \ref{lb4} we can define
\begin{align}
\wr^n\wtd{\mc S}\upphi(v_1,\dots,v_n,w_\blt)	\qquad\in\scr O\big(\mc D_{r_\blt\rho_\blt}^\times\times\Conf_{\wtd{\mc B}}(U_\blt\setminus\SXtd)\big)	\label{eq41}
\end{align}
whose restriction to each $\mc D_{r_\blt\rho_\blt}^\times\times\wtd{\mc C}_b^{\times n}$ is $\wr^n\wtd{\mc S}(\upphi|_b)(v_1,\dots,v_n,w_\blt)$.

Consider \eqref{eq41} at the level of $\scr O(\Conf_{\wtd{\mc B}}(U_\blt\setminus\SXtd))[[q_1^{\pm 1},\dots,q_M^{\pm 1}]]$. By applying Thm. \ref{lb21} to $\upphi|_b$, we see that the coefficients before $q_1,\dots,q_N$ of \eqref{eq40} and \eqref{eq41} agree when restricted to each $\wtd{\mc C}_b^{\times n}$. So $\eqref{eq40}=\eqref{eq41}$. In particular, \eqref{eq40} converges a.l.u. on $\mc D_{r_\blt\rho_\blt}^\times\times\Conf_{\wtd{\mc B}}(U_\blt\setminus\SXtd)$.
\end{rem}

\section{A geometric construction of permutation-twisted $\Vbb^{\otimes k}$-modules}\label{lb24}

Let $\Ubb$ be a (positive energy) VOA, and let $g$ be an automorphism of $\Ubb$ fixing the vacuum and the conformal vector of $\Ubb$. In particular, $g$ preserves the $L_0$-grading of $\Ubb$. We assume $g$ has finite order $k$. 

A (finitely-admissible)  \textbf{$g$-twisted $\Ubb$-module} is a vector space $\mc W$ together with a diagonalizable operator $\wtd L_0^g$, and an operation
\begin{gather*}
Y^g:\Ubb\otimes\mc W\rightarrow	\mc W[[z^{\pm 1/k}]]\\
u\otimes w\mapsto Y^g(u,z)w=\sum_{n\in \frac 1k \Zbb}Y^g(u)_nw\cdot z^{-n-1}
\end{gather*}
satisfying the following conditions:
\begin{enumerate}
\item $\mc W$ has $\wtd L_0^g$-grading $\mc W=\bigoplus_{n\in\frac 1k\Nbb}\mc W(n)$, each eigenspace $\mc W(n)$ is finite-dimensional, and for any $u\in\Ubb$ we have
\begin{align}
[\wtd L_0^g,Y^g(u)_n]=Y^g(L_0u)_n-(n+1)Y^g(u)_n.\label{eq43}	
\end{align}
In particular, for each $w\in\mc W$ the lower truncation condition follows: $Y^g(u)_nw=0$ when $n$ is sufficiently small.
\item $Y^g(\id,z)=\id_{\mc W}$.
\item ($g$-equivariance) For each $u\in\Ubb$,
\begin{align}
Y^g(gu,z)=Y^g(u,e^{-2\im\pi}z):=\sum_{n\in \frac 1k \Zbb}Y^g(u)_nw\cdot e^{2(n+1)\im\pi}z^{-n-1}.\label{eq45}
\end{align}
\item (Jacobi identity-analytic version) Let $\mc W'=\bigoplus_{n\in\frac 1k\Nbb}\mc W(n)^*$. Let $P_n$ be the projection of $\ovl{\mc W}=\coprod_{n\in\frac 1k\Nbb}\mc W(n)^*$ (the dual space of $\mc W'$) onto $\mc W(n)$ and similarly $\ovl{\Ubb}$ (the dual space of $\Ubb'$) onto $\Ubb(n)$. Then for each $u,v\in\Ubb,w\in\mc W,w'\in\mc W'$, and for each $z\neq \xi$ in $\Cbb^\times$ with chosen $\arg \xi$, the following series of $n$
\begin{gather}
\bk{Y^g(u,z)Y^g(v,\xi)w,w'}:=\sum_{n\in\frac 1k\Nbb}	\bk{Y^g(u,z)P_nY^g(v,\xi)w,w'}\label{eq50}\\
\bk{Y^g(v,\xi)Y^g(u,z)w,w'}:=\sum_{n\in\frac 1k\Nbb}	\bk{Y^g(v,\xi)P_nY^g(u,z)w,w'}\label{eq51}\\
\bk{Y^g(Y(u,z-\xi)v,\xi)w,w'}:=\sum_{n\in\Nbb}	\bk{Y^g(P_nY(u,z-\xi)v,\xi)w,w'}\label{eq52}
\end{gather}
(where $\xi$ is fixed) converge a.l.u. for $z$ in $|z|>|\xi|$, $|z|<|\xi|$, $|z-\xi|<|\xi|$ respectively. Moreover, for any fixed $\xi\in\Cbb^\times$ with chosen argument $\arg \xi$, let $R_{\xi}$ be the ray with argument $\arg \xi$ from $0$ to $\infty$, but with $0,\xi,\infty$ removed. Any point on $R_{\xi}$ is assumed to have argument $\arg \xi$. Then the above three expressions, considered as functions of $z$ defined on $R_{\xi}$ satisfying the three mentioned inequalities respectively, can be analytically continued to the same holomorphic function on the open set
\begin{align*}
	\Updelta_\xi=\Cbb\setminus\{\xi,-t\xi:t\geq 0\},
\end{align*}
which can furthermore be extended to a multivalued holomorphic function $f_{\xi}(z)$ on $\Cbb^\times\setminus\{\xi\}$ (i.e., a holomorphic function on the universal cover of $\Cbb^\times\setminus\{\xi\}$).
\end{enumerate}

In the above Jacobi identity, if we let the series $\sum_n h_n(z)$ be any of \eqref{eq50}, \eqref{eq51}, \eqref{eq52}, then by saying that this series converges a.l.u. for $z$ in an open set $\Omega$, we mean $\sup_{z\in K}\sum_n|f_n(z)|<+\infty$ for each compact $K\subset\Omega$; the $\sup$ is over all $z\in K$ with all possible $\arg z$.

\begin{rem}
The above analytic version of Jacobi identity is equivalent to the usual algebraic one (cf. \cite[Thm. 2.4]{Hua10}). Indeed, assume without loss of generality that $gu=e^{2\im j\pi/k}u$. Then the $g$-equivariance condition shows that $z^{\frac jk}Y^g(u,z)$ is single-valued over $z$. Thus,  $z^{\frac jk}$ times \eqref{eq50}, \eqref{eq51}, \eqref{eq52} are series expansions on $|z|>|\xi|$, $|z|<|\xi|$, $|z-\xi|<|\xi|$ respectively (not necessarily restricting to $R_\xi$) of the same single-valued holomorphic function $z^{\frac jk}f_\xi$ on $\Cbb^\times\setminus\{\xi\}$. By Strong Residue Theorem, this is equivalent to that for each $m,n\in\Zbb$,
\begin{align*}
\Big(\oint_{|z|=2|\xi|}-\oint_{|z|=|\xi|/3}-\oint_{|z-\xi|=|\xi|/3}\Big)z^{\frac jk+m}(z-\xi)^nf_\xi(z)dz=0,
\end{align*}
where in these integrals, $f_\xi(z)$ is replaced by \eqref{eq50}, \eqref{eq51}, \eqref{eq52} respectively. Equivalently,
\begin{align}
&\sum_{l\in\Nbb}{\frac jk+m\choose l}\big\langle Y^g\big(Y(u)_{n+l}v,\xi\big)w,w'\big\rangle\xi^{\frac jk+m-l}\nonumber	\\
=&\sum_{l\in\Nbb}{n\choose l}(-1)^l\big\langle Y^g(u)_{\frac jk+m+n-l}Y^g(v,\xi)w,w'\big\rangle\xi^l \nonumber\\
&-\sum_{l\in\Nbb}{n\choose l}(-1)^{n-l}\big\langle Y^g(v,\xi)Y^g(u)_{\frac jk+m+l}w,w'\big\rangle\xi^{n-l}.
\end{align}
By comparing the coefficients before $\xi^{-h-1}$, the above is equivalent to that for each $m,n\in\Zbb,h\in\frac 1k\Zbb$, (suppressing $w,w'$)
\begin{align}
&\sum_{l\in\Nbb}{\frac jk+m\choose l} Y^g\big(Y\big(u\big)_{n+l}~v\big)_{\frac jk+m+h-l}  \nonumber\\
=&\sum_{l\in\Nbb}{n\choose l}(-1)^l Y^g\big(u\big)_{\frac jk+m+n-l}Y^g\big(v\big)_{h+l}-\sum_{l\in\Nbb}{n\choose l}(-1)^{n-l}Y^g\big(v\big)_{n+h-l}Y^g\big(u\big)_{\frac jk+m+l}	
\end{align}
which is the algebraic Jacobi identity.
\end{rem}

\subsection*{Construction of twisted representations associated to cyclic permutation actions of $\Vbb^{\otimes k}$}

We let $\Ubb=\Vbb^{\otimes k}$ with conformal vector $\cbf\otimes\id\otimes\cdots\otimes\id+\cdots+\id\otimes\cdots\otimes\id\otimes\cbf$, and $g$ an automorphism defined by
\begin{align*}
g:(v_1,v_2,\dots,v_k)\in\Vbb^{\otimes k}\mapsto (v_k,v_1,\dots,v_{k-1}).	
\end{align*}
For each $\Vbb$-module with $\wtd L_0$-operator, we define an associated $g$-twisted $\Ubb$-module $\mc W$ as follows. 

As a vector space, $\mc W=\Wbb$. We define $\wtd L_0^g=\frac 1k\wtd L_0$.

Let $\zeta$ be the standard coordinate of $\Cbb$. Let $\fk X=(\Pbb^1;0,\infty)$. We associate to $0,\infty$ local coordinates local coordinates $\zeta,\zeta^{-1}$  and $\Vbb$-modules $\Wbb,\Wbb'$. Note
\begin{align*}
\mc U(\zeta,\zeta^{-1}):\scr W_{\fk X}(\Wbb\otimes\Wbb')\xrightarrow{\simeq}\Wbb\otimes\Wbb'	
\end{align*}
Let  $\bk{\cdot,\cdot}$ be  the pairing for $\Wbb$ and $\Wbb'$. We define a conformal block
\begin{gather*}
\uptau_\Wbb:\scr W_{\fk X}(\Wbb\otimes\Wbb')\rightarrow\Cbb,\\
\mc U(\eta_0,\eta_\infty)^{-1}(w\otimes w')\mapsto \bk{w,w'}
\end{gather*}
whenever the local coordinates $\eta_0,\eta_\infty$ at $0,\infty$ are such that $(\Pbb^1;0,\infty;\eta_0,\eta_\infty)\simeq(\Pbb^1;0,\infty;\zeta,\zeta^{-1})$. It is easy to see that this definition is independent of the choice of such $\eta_0,\eta_\infty$. 

In the setting of Thm. \ref{lb17}, we have
\begin{gather*}
\wr^k\uptau_\Wbb:\underbrace{\scr V_{\fk X}(\Cbb^\times)\otimes\cdots \otimes\scr V_{\fk X}(\Cbb^\times)}_k\otimes\scr W_{\fk X}(\Wbb\otimes\Wbb')\rightarrow\scr O(\Conf^k(\Cbb^\times))
\end{gather*}
where all the $\otimes$ are over $\Cbb$. Let  \index{zz@$\omega_k=e^{-2\im\pi/k}$}
\begin{align*}
\omega_k=e^{-2\im\pi/k}.	
\end{align*}
Since $\zeta^k:z\mapsto z^k$ is locally injective holomorphic on $\Cbb^\times$, we have a trivilization
\begin{align*}
\mc U_\varrho(\zeta^k):	\scr V_{\fk X}|_{\Cbb^\times}\xrightarrow{\quad\simeq\quad}\Vbb\otimes_\Cbb\scr O_{\Cbb^\times}.
\end{align*}
Then, for each $w\in\Wbb,w'\in\Wbb'$, and for each $v_1,\dots,v_n\in\Vbb$ (considered as a constant section of $\Vbb\otimes_\Cbb\scr O(\Cbb^\times)$) we define, for $v_\blt=v_1\otimes\cdots\otimes v_k\in\Vbb^{\otimes k}$,
\begin{equation}
\boxed{
\begin{array}{l}
~~~~~\bk{Y^g(v_\blt,z)w,w'}\\[0.8ex]
=\wr^k\uptau_\Wbb\big(\mc U_\varrho(\zeta^k)^{-1}v_1,\dots,\mc U_\varrho(\zeta^k)^{-1}v_k,\mc U(\zeta,\zeta^{-1})^{-1}(w\otimes w')\big)\Big|_{\omega_k^{\blt-1}\sqrt[k]{z}}	
\end{array}		\label{eq42}
}
\end{equation}
where, for each $z\in\Cbb^\times$ with argument $\arg z$, \index{zz@$\omega_k^{\blt-1}\sqrt[k]{z}$}
\begin{align}
\omega_k^{\blt-1}\sqrt[k]{z}:=(\sqrt[k]{z},\omega_k\sqrt[k]{z},\omega_k^2\sqrt[k]{z},\dots,\omega_k^{k-1}\sqrt[k]{z})	\qquad \in\Conf^k(\Cbb^\times),\label{eq48}
\end{align}
and $\sqrt[k]z$ is assumed to have argument $\frac 1k\arg z$.

\eqref{eq42} is a multi-valued function of $z$, single-valued of $\sqrt[k]{z}\in\Cbb^\times$. So we have Laurent series expansion
\begin{align*}
\bk{Y^g(v_\blt,z)w,w'}=\sum_{n\in\frac 1k\Zbb}\bk{Y^g(v_\blt)_nw,w'}	z^{-n-1}
\end{align*}
which defines $Y^g(v_\blt)_n$ as a linear map $\Wbb\otimes\Wbb'\rightarrow\Cbb$.

\begin{lm}
Each $Y^g(v_\blt)_n$ is a linear operator on $\Wbb$. Moreover, \eqref{eq43} is satisfied.
\end{lm}

\begin{proof}
For each $q\in\Cbb^\times$ with chosen $\arg q$, by \eqref{eq11} we have
\begin{align*}
\mc U(q^{\frac 1k}\zeta,q^{-\frac 1k}\zeta^{-1})\mc U(\zeta,\zeta^{-1})^{-1}=q^{\frac 1k\wtd L_0}\otimes q^{-\frac 1k\wtd L_0}=q^{\wtd L_0^g}\otimes q^{-\wtd L_0^g}.	
\end{align*}
Thus
\begin{align}
&\bk{Y^g(v_\blt,z)q^{-\wtd L_0^g}w,q^{\wtd L_0^g}w'}\nonumber\\
=&\wr^k\uptau_\Wbb\big(\mc U_\varrho(\zeta^k)^{-1}v_1,\dots,\mc U_\varrho(\zeta^k)^{-1}v_k,\mc U(q^{\frac 1k}\zeta,q^{-\frac 1k}\zeta^{-1})^{-1}(w\otimes w')\big)\Big|_{\omega_k^{\blt-1}\sqrt[k]z}.	\label{eq44}
\end{align}
We have an equivalence of pointed Riemann spheres with locally injective functions and local coordinates (at the last two marked points)
\begin{align*}
&(\Pbb^1;\omega_k^{\blt-1}\sqrt[k]z,0,\infty;\zeta^k,q^{\frac 1k}\zeta,q^{-\frac 1k}\zeta^{-1})\\
\simeq&	(\Pbb^1;\omega_k^{\blt-1}\sqrt[k]{qz},0,\infty;q^{-1}\zeta^k,\zeta,\zeta^{-1})
\end{align*}
defined by $z\in\Pbb^1\mapsto \sqrt[k]{q}z\in\Pbb^1$, where $\sqrt[k]{q}$ has argument $\frac 1k\arg q$. By \eqref{eq12} and \eqref{eq10}, on $\Vbb$ we have
\begin{align*}
	\mc U_\varrho(\zeta^k)\mc U_\varrho(q^{-1}\zeta^k)^{-1}=\mc U(\varrho(\zeta^k|q^{-1}\zeta^k))=q^{L_0}.	
\end{align*}
So \eqref{eq44} equals
\begin{align*}
&\wr^k\uptau_\Wbb\big(\mc U_\varrho(q^{-1}\zeta^k)^{-1}v_1,\dots,\mc U_\varrho(q^{-1}\zeta^k)^{-1}v_k,\mc U(\zeta,\zeta^{-1})^{-1}(w\otimes w')\big)\Big|_{\omega_k^{\blt-1}\sqrt[k]{qz}}\\
=&\wr^k\uptau_\Wbb\big(\mc U_\varrho(\zeta^k)^{-1}q^{L_0}v_1,\dots,\mc U_\varrho(\zeta^k)^{-1}q^{L_0}v_k,\mc U(\zeta,\zeta^{-1})^{-1}(w\otimes w')\big)\Big|_{\omega_k^{\blt-1}\sqrt[k]{qz}}
\end{align*}
We conclude
\begin{align*}
\bk{Y^g(v_\blt,z)q^{-\wtd L_0^g}w,q^{\wtd L_0^g}w'}=\bk{Y^g(q^{L_0}v_\blt,qz)w,w'}.	
\end{align*}
So, if $L_0 v_\blt=\alpha v_\blt$, $\wtd L_0^gw=\beta w$, $\wtd L_0^g w'=\gamma w'$, then
\begin{align*}
\bk{Y^g(v_\blt,z)w,w'}=q^{\alpha+\beta-\gamma}\bk{Y^g(v_\blt,qz)w,w'},	
\end{align*}
which shows, by looking at the coefficients before  $z^{-n-1}$, that $\bk {Y^g(v_\blt)_nw,w'}$ equals $0$ unless $\alpha+\beta-\gamma-n-1=0$. This proves $Y^g(v_\blt)_n\mc W(\beta)\subset \mc W(\alpha+\beta-n-1)$. In particular, $Y^g(v_\blt)_n$ can be regarded as a linear operator on $\mc W$.
\end{proof}

Using part (3) and (4) of Thm. \ref{lb17}, it is easy to show $Y^g(\id,z)=\id_{\mc W}$ and show \eqref{eq45}. Moreover:

\begin{thm}
$Y^g$ satisfies the Jacobi identity. Therefore, $(\mc W,Y^g)$ is a $g$-twisted $\Vbb^{\otimes k}$-module.
\end{thm}

\begin{proof}
Choose the two vectors of $\Ubb$ to be $u_\blt=u_1\otimes\cdots\otimes u_k,v_\blt=v_1\otimes\cdots\otimes v_k\in\Vbb^{\otimes k}$. Identify $\scr W_{\fk X}(\Wbb\otimes\Wbb')=\Wbb\otimes\Wbb'$ via $\mc U(\zeta,\zeta^{-1})$. Identify $\scr V_{\fk X}|_{\Cbb^\times}=\Vbb\otimes_\Cbb\scr O_{\Cbb^\times}$ via $\mc U_\varrho(\zeta^k)$. For each $\xi\in\Cbb^\times$ with chosen $\arg \xi$, we define
\begin{align}
f_\xi(z)=\wr^{2k}\uptau_\Wbb(u_1,\dots,u_k,v_1,\dots,v_k,w\otimes w')\Big|_{\omega_k^{\blt-1}\sqrt[k]{z},~\omega_k^{\blt-1}\sqrt[k]{\xi}}\label{eq47}
\end{align}
where $\omega_k^{\blt-1}\sqrt[k]{\xi}$ is a $k$-tuple understood in a similar way as \eqref{eq48}. Then $f_\xi$ is a multivalued holomorphic function which lifts to a single-valued one on the  $k$-fold covering space  $\Cbb^\times\setminus(\omega_k^{\blt-1}\sqrt[k]{\xi})$ of $\Cbb^\times\setminus\{\xi\}$.

Let $(m_{n,\alpha})_{\alpha\in\fk A}$ be a set of basis of $\Wbb(n)$ with dual basis $(\wch m_{n,\alpha})_{\alpha\in\fk A}$. Assume $0<|z|<|\xi|$. We shall show that the following infinite sum over $n$
\begin{align}
&\bk{Y^g(v_\blt,\xi)Y^g(u_\blt,z)w,w'}\nonumber\\
=&\sum_{n\in\Nbb}\sum_{\alpha\in\fk A}\wr^k\uptau_\Wbb(u_1,\dots,u_k,w\otimes \wch m_{n,\alpha})_{\omega_k^{\blt-1}\sqrt[k]{z}}\cdot 	\wr^k\uptau_\Wbb(v_1,\dots,v_k,m_{n,\alpha}\otimes w')_{\omega_k^{\blt-1}\sqrt[k]{\xi}}\label{eq46}
\end{align}
converges a.l.u. to $f_\xi(z)$. Indeed, this expression is the sewing at $q=1$ of the $2k$-propagation of the conformal block
\begin{gather*}
\upphi:\Wbb\otimes\Wbb'\otimes\Wbb\otimes\Wbb'\rightarrow\Cbb,\\ w_1\otimes w_1'\otimes w_2	\otimes w_2'\mapsto \bk{w_1,w_1'}\cdot \bk{w_2,w_2'}
\end{gather*}
associated to $(\Pbb^1_a\sqcup\Pbb^1_b;0_a,\infty_a,0_b,\infty_b)$. Here, $\Pbb^1_a,\Pbb^1_b$ are two identical Riemann spheres.  The sewing is along $\infty_a$ and $0_b$ using local coordinates $\zeta,\zeta^{-1}$, and by choosing suitable open discs $W'\ni\infty_a,W''\ni 0_b$ with radius $r,\rho$ satisfying $r\rho>1$ such that $W',W''$ do not intersect $\omega_k^{\blt-1}\sqrt[k]{z}$ and $\omega_k^{\blt-1}\sqrt[k]{\xi}$. (Note that $|z|<|\xi|$ guarantees the existence of such $W',W''$.) Since the sewing of $\upphi$ clearly converges a.l.u. on $\mc D_{r\rho}^\times$, by Thm. \ref{lb21}, the sewing at $q=1$ of $\wr^{2k}\upphi$ (which is \eqref{eq46}) converges a.l.u. (for varying $z$) to the $2k$-propagation of the sewing, which is just $f_\xi(z)$. A similar argument shows that when $0<|\xi|<|z|$, $\bk{Y^g(u_\blt,z)Y^g(v_\blt,\xi)w,w'}$ converges a.l.u. (for varying $z$) to $f_\xi(z)$.

Consider $g_\xi\in\Conf^k(\Cbb\setminus\omega_k^{\blt-1}\sqrt[k]{\xi})$ defined by
\begin{align*}
g_\xi(z_1,\dots,z_k)=\wr^{2k}\uptau_\Wbb(u_1,\dots,u_k,v_1,\dots,v_k,w\otimes w')\Big|_{z_1,\dots,z_k,~\omega_k^{\blt-1}\sqrt[k]{\xi}}.	
\end{align*}
The region $\Omega=\{z\in\Cbb^\times:|z^k-\xi|<|\xi|\}$ has $k$ connected components $\Omega_1,\dots,\Omega_k$, each one $\Omega_i$ contains exactly one element $\omega_k^{i-1}\sqrt[k]{\xi}$ of $\omega_k^{\blt-1}\sqrt[k]{\xi}$, and $\Omega_i\simeq\zeta^k(\Omega_i)$ where $\zeta^k(\Omega_i)$ is the open disc with center $\xi$ and radius $|\xi|$. By Thm. \ref{lb17} and the definition \eqref{eq42}, whenever $z_i\in\Omega_i$ for each $i$, we have  (letting $x_1,\dots,x_k$ be formal variables)
\begin{align}
&g_\xi(z_1,\dots,z_k)\nonumber\\
=&\wr^k\uptau_\Wbb(Y(u_1,x_1)v_1,\dots,Y(u_k,x_k)v_k,w\otimes w')\Big|_{\omega_k^{\blt-1}\sqrt[k]{\xi}}\Big|_{x_k=z_k^k-\xi}\cdots \Big|_{x_1=z_1^k-\xi}\nonumber\\
=&\bk{Y^g(Y(u_1,x_1)v_1\otimes\cdots\otimes Y(u_k,x_k)v_k,\xi)w,w'}\Big|_{x_k=z_k^k-\xi}\cdots \Big|_{x_1=z_1^k-\xi}.\label{eq49}
\end{align}
where the right hand side converges absolutely and successively for $x_k,x_{k-1},\dots,x_1$. Since the simultaneous Laurent series expansion of the holomorphic function $h(\varkappa_1,\dots\varkappa_k)=g_\xi(\sqrt[k]{\xi+\varkappa_1},\omega_k\sqrt[k]{\xi+\varkappa_2},\dots,\omega_k^{k-1} \sqrt[k]{\xi+\varkappa_k})$ in the region $0<|\varkappa_i|<|\xi|$ (for all $i$) clearly converges a.l.u., and since the coefficients of these series agree with those before the powers of $x_1,\dots,x_k$ on the right hand side of \eqref{eq49} (by taking Laurent series expansion through contour integrals), we see that \eqref{eq49} converges absolutely (as a multi-variable series) to $g_\xi(z_1,\dots,z_k)$  at the desired points. 

Now we  assume $0<|z-\xi|<|\xi|$, assume  $\arg z$ is such that $\sqrt[k]{z}\in\Omega_1\ni\sqrt[k]{\xi}$ (which is true when $\arg z=\arg \xi$), and  set $(z_1,\dots,z_k)=\omega_k^{\blt-1}\sqrt[k]{z}$. Then we see that $\bk{Y^g(Y(u_\blt,z-\xi)v_\blt,\xi)w,w'}$ converges a.l.u. to $g_\xi(\omega_k^{\blt-1}\sqrt[k]{z})=f_\xi(z)$. This finishes the verification of the Jacobi identity.
\end{proof}

\begin{rem}
Using Thm. \ref{lb17}, it is easy to see that
\begin{align*}
\wr^k\uptau_\Wbb(\id,\cdots,\mc U_\varrho(\zeta)^{-1}v_i,\cdots,\id,w\otimes w')|_z=\bk{Y(v,z)w,w'}.
\end{align*}
By \eqref{eq12}, $\mc U_\varrho(\zeta)\mc U_\varrho(\zeta^k)^{-1}=\mc U(\varrho(\zeta|\zeta^k))$. Thus, when $v_\blt=v_1\otimes\id\otimes\cdots\otimes\id$, \eqref{eq42} becomes
\begin{align*}
\bk{Y(\mc U(\varrho(\zeta|\zeta^k)_{\sqrt[k]z})v_1,\sqrt[k]z)w,w'}.	
\end{align*}
By \eqref{eq10}, $\varrho(\zeta|\zeta^k)_{\sqrt[k]z}$ sends $z_1^k-z$ to $z_1-\sqrt[k]z$ when $z_1$ is close to $\sqrt[k]z$. Hence this transformation equals $\delta_{k,z}$ where 
\begin{align*}
\delta_{k,z}(t)=(z+t)^{\frac 1k}-z^{\frac 1k}.	
\end{align*}
We conclude
\begin{align}
Y^g(v_1\otimes\id\otimes\cdots\otimes \id,z)=Y(\mc U(\delta_{k,z})v_1,\sqrt[k]z).	
\end{align}
This equation uniquely determines the $g$-twisted module structure of $\mc W$, since $\Vbb^{\otimes k}$ is $g$-generated by vectors of the form $v_1\otimes\id\otimes\cdots\otimes \id$.

It is not hard to check that $\mc U(\delta_{k,z})$ agrees with the operator $\Delta_k(z)$ in \cite{BDM02}. Thus, our $g$-twisted module $(\mc W,Y^g)$ agrees with the one $(T^k_g(\Wbb),Y_g)$ in \cite[Thm. 3.9]{BDM02}.
\end{rem}

\appendix

\section{Strong residue theorem for analytic families of curves}\label{lb6}

Let $\fk X=(\pi:\mc C\rightarrow\mc B;\sgm_1,\dots,\sgm_N)$ be a (holomorphic) family of $N$-pointed compact Riemann surfaces. Recall the definition in Sec. \ref{lb2}. In particular, we assume each connected component of each fiber $\mc C_b=\pi^{-1}(b)$ contains at least one of $\sgm_1(b),\dots,\sgm_N(b)$. We let $\scr E$ be a holomorphic vector bundle on $\mc C$ with finite rank, and let $\scr E^\vee$ be its dual bundle.

We assume that $\fk X$ is equipped with local coordinates $\eta_1,\dots,\eta_N$ at $\sgm_1(\mc B),\dots,\sgm_N(\mc B)$ respectively. Assume for each $j$ that  $\eta_j$ is defined on a neighborhood $W_j\subset\mc C$ of $\sgm_j(\mc B)$ which intersects only $\sgm_j(\mc B)$ among $\sgm_1(\mc B),\dots,\sgm_N(\mc B)$, and that there is a  trivialization
\begin{align*}
\scr E_j|_{W_j}\simeq E_j\otimes_\Cbb\scr O_{W_j}	
\end{align*}
with dual trivialization
\begin{align*}
\scr E_j^\vee|_{W_j}\simeq E_j^\vee\otimes_\Cbb\scr O_{W_j},	
\end{align*}
where $E_j$ is a finite-dimensional vector space and $E_j^\vee$ is its dual space. We identify $\scr E_j|_{W_j}$ and $\scr E^\vee|_{W_j}$ with their trivializations.

For each $j$, we identify
\begin{align*}
W_j=(\pi,\eta_j)(W_j)	\qquad\text{via }(\pi,\eta_j).
\end{align*}
Then $W_j$ is a neighborhood of $\mc B\times\{0\}$ in $\mc B\times\Cbb$. We let $z$ be the standard coordinate of $\Cbb$. Consider
\begin{gather}
s_j=\sum_{n\in\Zbb} e_{j,n}\cdot z^n\qquad \in\big(E_j\otimes_\Cbb\scr O(\mc B) \big)((z)),\label{eq19}	
\end{gather}
where each $e_{j,n}\in E_j\otimes_\Cbb\scr O(\mc B)$ is $0$ when $n$ is sufficiently small. Considering $e_{j,n}$ as an $E_j$-valued holomorphic on $\scr O(\mc B)$, we let $e_{j,n}(b)\in E_j$ be its value at $b\in\mc B$. Then $s_j(b)$, the restriction of $s_j$ to $\mc C_b$, is represented by
\begin{align*}
s_j(b)=\sum_n e_{j,n}(b)z^n\qquad\in E_j((z)).	
\end{align*} 
Suppose that $s$ is a section of $\scr E(\star\SX)$ defined on $W_j$. Then $s|_{W_j}=s|_{W_j}(b,z)$ is an $E_j$-valued meromorphic function on $W_j$ with poles at $z=0$. We say that $s$ \textbf{has series expansion $s_j$ at $\sgm_j(\mc B)$} if for each $b\in\mc B$, the meromorphic function $s|_{W_j}(b,z)$ of $z$ has Laurent series expansion \eqref{eq19} at $z=0$.

For each $b\in\mc B$, choose $\sigma_b\in H^0(\mc C_b,\scr E^\vee|_{\mc C_b}\otimes\omega_{\mc C_b}(\star S_{\fk X_b}))$. Then in $W_{j,b}=W_j\cap\pi^{-1}(B)$, $\sigma_b$ can be regarded as an $E_j^\vee\otimes dz$-valued holomorphic function but with possibly finite poles at $z=0$. So it has series expansion at $z=0$:
\begin{align*}
\sigma_b|_{W_{j,b}}(z)=\sum_n\phi_{j,n}z^ndz\qquad \in E_j^\vee((z))dz	
\end{align*}
where $\phi_{j,n}\in E_j^\vee$. We define the residue pairing
\begin{align}
\Res_j \bk{s_j,\sigma_b}=&\Res_{z=0}\langle s_j(b),\sigma_b|_{U_j,b}(z) \rangle\nonumber\\
=&\Res_{z=0}\bigg(\Big\langle \sum_n e_{j,n}(b)z^n, \sum_n \phi_{j,n}z^n  \Big\rangle dz\bigg).
\end{align} 
in which the pairing between $E_j$ and $E_j^\vee$ is denoted by $\bk{\cdot,\cdot}$.

We now prove the Strong Residue Theorem for $\scr E$. Our proof is inspired by that of \cite[Thm. 1.22]{Ueno08}.

\begin{thm}\label{lb7}
For each $1\leq j\leq N$, choose $s_j$ as in \eqref{eq19}.  Then the following statements are equivalent.
	
(a) There exists $s\in H^0(\mc C,\scr E(\star\SX))$ whose series expansion at $\sgm_j(\mc B)$ (for each $1\leq j\leq N$) is $s_j$.

(b) For each $b\in\mc B$, there exists $s_b\in H^0(\mc C_b,\scr E|_{\mc C_b}(\star\SXb))$ whose series expansion at $\sgm_j(b)$ (for each $1\leq j\leq N$) is $s_j(b)$.
	
(c) For any $b\in\mc B$ and any $\sigma_b\in H^0\big(\mc C_b,\scr E^\vee|_{\mc C_b}\otimes\omega_{\mc C_b}(\star S_{\fk X_b})\big)$,
	\begin{align}
		\sum_{j=1}^N\Res_j\bk{s_j,\sigma_b}=0.
	\end{align}
	Moreover, when these statements hold, there is only one $s\in H^0(\mc C,\scr E(\star\SX))$ satisfying (a).
\end{thm}

\begin{proof}
(a) trivially implies (b). That (b) implies (c) follows from Residue theorem (i.e., Stokes theorem): The evaluation between $s_b$ and $\sigma_b$ is an element of $H^0(\mc C_b,\omega_{\mc C_b}(\star S_{\fk X_b}))$ whose total residue over all poles is $0$.
	
If $s$ satisfies (a), then for each $b\in\mc B$, $s|_{\mc C_b}$ is uniquely determined by its series expansions near $\sgm_1(b),\dots,\sgm_N(b)$ (since each component of $\mc C_b$ contains some $\sgm_j(b)$). Therefore the sections satisfying (a) is unique.
	
Now assume (c) is true. We shall prove (a).  Suppose that for each $b\in\mc B$ we can find a neighborhood $V\subset\mc B$ such that an $s$ satisfying (a) exists for the family $\fk X_V$. Then, by the uniqueness proved above, we can glue all these locally defined $s$ to a global one. Thus, we may shrink $\mc B$ to a small neighborhood of a given $b_0\in\mc B$ when necessary. 

We first note that, by replacing $\mc B$ with a neighborhood of a given $b_0\in\mc B$, we may assume $\pi_*\scr E(-k\SX)=0$ for sufficiently large $k$. Indeed, choose any $b_0\in\mc B$. Then by Serre duality,
\begin{align}
H^0\big(\mc C_b,\scr E|_{\mc C_b}(-kS_{\fk X_b})\big)\simeq H^1\big(\mc C_b,\scr E^\vee|_{\mc C_b}\otimes\omega_{\mc C_b}(kS_{\fk X_b})\big),	\label{eq21}
\end{align}
which, by Serre vanishing theorem, equals $0$ for some $k=k_0$ when $b=b_0$. Since $\pi$ is open, $\fk X$ is a flat family (\cite[Thm. II.2.13]{GPR} or \cite[Sec. 3.20]{Fis76}). Thus, we can apply the upper-semicontinuity theorem (\cite[Thm. III.4.7]{GPR} or \cite[Thm. III.4.12]{BS76}) to see that \eqref{eq21} vanishes for $k=k_0$ and (by shrinking $\mc B$ to a neighborhood of $b_0$) any $b\in\mc B$. Since the vector space $H^0\big(\mc C_b,\scr E|_{\mc C_b}(-kS_{\fk X_b})\big)$ shrinks as $k$ increases, \eqref{eq21} is constantly zero for all $b\in\mc B$ and $k\geq k_0$. This implies $\pi_*\scr E(-k\SX)=0$ for all $k\geq k_0$ (\cite[Thm. III.4.7-(d)]{GPR} or \cite[Cor. III.3.5]{BS76}).
	
Choose $p\in\mathbb N$ such that for each $1\leq j\leq N$, the $e_{j,n}$ in \eqref{eq19} equals $0$ when $n<-p$. For any $k\geq k_0$, as $\pi_*\scr E(-k\SX)=0$, the short exact sequence
\begin{align*}
0\rightarrow \scr E(-k\SX)\rightarrow \scr E(p\SX)\rightarrow \scr E(p\SX)/\scr E(-k\SX)\rightarrow 0
\end{align*}
induces a long one
\begin{align}
0\rightarrow\pi_*\scr E(p\SX)\rightarrow \pi_*\big(\scr E(p\SX)/\scr E(-k\SX)\big)\xrightarrow{\delta}R^1\pi_*\scr E(-k\SX).\label{eq22}
\end{align}

For each $1\leq j\leq N$, set $s_j|_k=\sum_{n<k}e_{j,n}\cdot z^n$, which can be regarded as a section in $\scr E(p\SX)(W_j)$. Let $W_0=\mc C\setminus\SX$. Then $\fk U=\{W_0,W_1,\dots,W_N \}$ is an open cover of $\mc C$. Define \v{C}ech $0$-cocycle $\psi=(\psi_j)_{0\leq j\leq N}\in Z^0(\fk U,\scr E(p\SX)/\scr E(-k\SX))$ by setting
\begin{align*}
\psi_0=0,\qquad\psi_j=s_j|_k \quad(1\leq j\leq N).
\end{align*}
Then $\delta\psi=\big((\delta\psi)_{i,j}\big)_{0\leq i,j\leq N}\in Z^1(\fk U,\scr E(-k\SX))$ is described as follows: $(\delta\psi)_{0,0}=0$; if $i,j>0$ then $(\delta\psi)_{i,j}$ is not defined since $W_i\cap W_j=\emptyset$; if $1\leq j\leq N$ then $(\delta\psi)_{j,0}=-(\delta\psi)_{0,j}$ equals $s_j|_k$ (considered as a section in $\scr E(-k\SX)(W_j\cap W_0)$).
	
Consider $\delta\psi$ as a section of $R^1\pi_*\scr E(-k\SX)$. We shall show that $\delta\psi=0$.  By the fact that \eqref{eq21} vanishes and the invariance of Euler characteristic,  $\dim H^1\big(\mc C_b,(\scr E|_{\mc C_b})(-kS_{\fk X_b})\big)$ is locally constant over $b\in\mc B$, which shows  that $R^1\pi_*(\mc C,\scr E(-k\SX))$ is locally free and its fiber at $b$ is naturally equivalent to $H^1\big(\mc C_b,(\scr E|_{\mc C_b})(-kS_{\fk X_b})\big)$. (Cf. \cite[Thm. III.4.7]{GPR}  or \cite[Thm. III.4.12]{BS76}.) Thus, it suffices to show that for each fiber $\mc C_b$, the restriction $\delta\psi|_{\mc C_b}\in H^1(\mc C_b,\scr E|_{\mc C_b}(-k S_{\fk X_b}))$ is zero.

The residue pairing for the Serre duality
\begin{align*}
H^1(\mc C_b,\scr E|_{\mc C_b}(-k\SX))\simeq H^0\big(\mc C_b,\scr E^\vee|_{\mc C_b}\otimes\omega_{\mc C_b}(kS_{\fk X_b})\big)^*
\end{align*}
applied to $\delta\psi|_{\mc C_b}$ and any $\sigma_b\in H^0\big(\mc C_b,\scr E^\vee|_{\mc C_b}\otimes\omega_{\mc C_b}(kS_{\fk X_b})\big)$, is given by
\begin{align*}
\bk{\delta\psi|_{\mc C_b},\sigma_b}=\sum_{j=1}^N\Res_j\bk{s_j|_k,\sigma_b}.
\end{align*}
Since for each $1\leq j\leq N$, $\bk{s_j-s_j|_k,\sigma_b}$ has removable singularity at $z=0$, we have $\Res_j\bk{s_j-s_j|_k,\sigma_b}=0$. Therefore, 
\begin{align*}
\bk{\delta\psi|_{\mc C_b},\sigma_b}=\sum_{j=1}^N\Res_j\bk{s_j,\sigma_b}=0.
\end{align*}
Thus $\delta\psi|_{\mc C_b}=0$ for any $b$.  This proves that $\delta\psi=0$. 
	
By \eqref{eq22}, for each $k\geq k_0$, there is a unique $s|_k\in\big(\pi_*\scr E(p\SX)\big)(\mc B)= H^0(\mc C,\scr E(p\SX))$ which is sent to $\psi\in \pi_*\big(\scr E(p\SX)/\scr E(-k\SX)\big)(\mc B)$. So near $\sgm_j(\mc B)$, $s|_k$ has series expansion
\begin{align}
s|_k=s_j|_k+\bullet z^k+\bullet z^{k+1}+\cdots.\label{eq23}
\end{align}
By this uniqueness, we must have $s|_{k_0}=s|_{k_0+1}=s|_{k_0+2}=\cdots$. Let $s=s|_{k_0}$. Then $s$ has series expansion $s_j$ at $\sgm_j(\mc B)$ for each $j$.
\end{proof}

We remark that the above proof also applies to locally free sheaves over a proper flat family of pointed complex curves (with at worst nodal singularities) such that each $S_{\fk X_b}$ does not intersect the node of $\mc C_b$, and that $S_{\fk X_b}$ intersects each irreducible component of $\mc C_b$. This is because the residue pairing for Serre duality is described in the same way as in the smooth case.

\printindex	
	
	\noindent {\small \sc Yau Mathematical Sciences Center, Tsinghua University, Beijing, China.}
	
	\noindent {\textit{E-mail}}: binguimath@gmail.com\qquad bingui@tsinghua.edu.cn
\end{document}